\documentclass[reqno,11pt]{amsart}

\usepackage{times,amsmath,cancel,stmaryrd,graphicx}
\usepackage{amsfonts,enumitem,amssymb,color,mathrsfs}
\usepackage[colorlinks=true]{hyperref}
\usepackage{cleveref,enumerate}
\usepackage{lmodern}
\usepackage[dvipsnames]{xcolor}
\usepackage{comment}
\usepackage{tikz}
\usetikzlibrary{arrows.meta,calc}
\usepackage{cite}

\usepackage{float}

\newtheorem{thm}{Theorem}[section]
\newtheorem{lem}[thm]{Lemma}
\newtheorem{cor}[thm]{Corollary}
\newtheorem{prop}[thm]{Proposition}

\newtheorem{rem}[thm]{Remark}

\numberwithin{equation}{section}

\newcommand{\aequation}{\renewcommand{\theequation}{\mbox{A.\arabic{equation}}}}

\newcommand{\nequation}{\setcounter{equation}{0}}

\newcommand{\R}{\mathbb{R}}

\newcommand{\E}{\mathbb{E}}

\newcommand{\ve}{\varepsilon}
\newcommand{\rd}{\mathrm{d}}

\newcommand{\dhr}{\mathrel{\lhook\joinrel\relbar\kern-.8ex\joinrel\lhook\joinrel\rightarrow}}

\allowdisplaybreaks

\begin{document}

\title[A CSTR Biofilm Model]{Analysis of a Biofilm Model in a Continuously Stirred Tank Reactor with Wall Attachment}


%
\author{Katerina Nik}
\address{King Abdullah University of Science and Technology (KAUST)\\
CEMSE Division\\
Thuwal 23955-6900\\
Saudi Arabia}
\email{katerina.nik@kaust.edu.sa}

\author{Christoph Walker}
\address{Leibniz Universit\"at Hannover\\
Institut f\"ur Angewandte Mathematik\\
Welfengarten 1\\
30167 Hannover\\
Germany}
\email{walker@ifam.uni-hannover.de}
\date{\today}

\begin{abstract}
We investigate a mathematical model for a bacterial population in a continuously stirred tank reactor with wall attachment. The model couples a free-boundary value problem for substrate diffusion in the one-dimensional biofilm
with a system of nonlinear ODEs for biofilm thickness, suspended biomass, and free substrate concentration. We establish global well-posedness and analyze the long-term dynamics. In particular, we characterize the local and global stability of the trivial (washout) equilibrium, prove the existence of a nontrivial equilibrium, and, under additional structural assumptions, establish its uniqueness and derive conditions for its local stability.
\end{abstract}
%
\maketitle
\section{Introduction}

Microbial communities have become a central topic of interdisciplinary research due to their fundamental role in natural, medical, and industrial processes. In many environments, microorganisms do not persist as isolated planktonic cells but instead form spatially structured aggregates, commonly referred to as biofilms. On submerged surfaces in aqueous systems, cells may attach to the boundary, transition to a sessile state, and develop layered structures.

In numerous engineered systems, such as continuously stirred tank reactors (CSTRs), suspended and surface-associated microbial populations coexist. These compartments are dynamically coupled: planktonic cells may attach to an established biofilm, while sessile cells may detach and re-enter the bulk phase. Although attachment and detachment processes are known to play a crucial role in microbial population dynamics, the underlying mechanisms remain only partially understood.

Despite extensive experimental investigations and application-oriented modeling efforts, comparatively little attention has been devoted to the systematic derivation and rigorous mathematical analysis of models capturing the coupled dynamics of suspended and attached populations. Motivated by~\cite{MasicEberl12}, we study a mathematical model describing a bacterial population in a continuously stirred tank reactor with wall attachment.
The model extends the classical Freter model~\cite{Freter,Jones03,BallykJonesSmith,Stemmons} by incorporating additional dynamical features. In particular, the sessile population is represented by a one-dimensional biofilm whose evolution follows established biofilm modeling frameworks~\cite{Jones03,WannerGujer86,WannerEtAl06,KD10}. The reactor is assumed to be continuously supplied with fresh medium. The biofilm uniformly covers the available surface for colonization, and growth-limiting substrate diffuses into the biofilm layer, where spatial gradients arise as a consequence of diffusion and microbial consumption.

More precisely, the model proposed in~\cite{MasicEberl12} is of the form
\begin{subequations}\label{C}
\begin{align}
\kappa\partial_z^2 c&=r(c)\,, &&0<z<h(t)\,, \,\quad t>0\,, \label{C1}\\
\partial_z c(0)&=0\,, \quad  c\big(h(t)\big)=S(t)\,, && t>0\,,\label{C2}\\
h'&=\int_0^{h(t)} g\big(r(c(t,z))\big)\,\rd z +\frac{\alpha}{\beta} Q(t)-d(h)h\,,&& t>0\,, \label{C3}\\
S'&=D(S^*-S)-k_1Q(t)\nu(S)-\rho\partial_zc\big(t,h(t)\big)\,, && t>0\,,\label{C4}\\
Q'&= \big(\nu(S(t))-D-k_Q\big)Q+\beta d\big(h(t)\big)h(t)-\alpha Q\,, && t>0\,,\label{C5}
\end{align}
subject to the nonnegative initial conditions
\begin{align}
h(0)&=h_0\,,\quad S(0)=S_0\,,\quad Q(0)=Q_0 \label{c7}\,,
\end{align}
\end{subequations}
for the unknown substrate concentration $c=c(t,z)$ in the biofilm at thickness $z$, the biofilm thickness $h=h(t)$, the free substrate concentration $S=S(t)$ in the bulk,  and the suspended biomass $Q=Q(t)$ in the bulk.

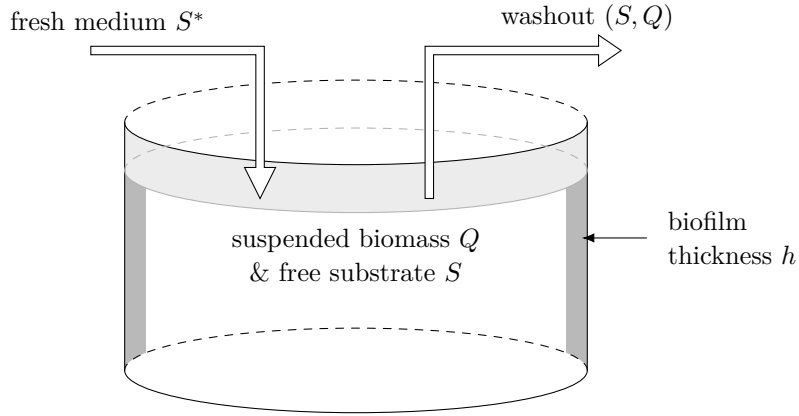
\begin{figure}[h]
\centering\hspace*{1.2cm}
\begin{tikzpicture}[line cap=round,line join=round,>=Latex,scale=0.75]

\def\W{8.2}          
\def\H{4.4}          
\def\E{0.75}         
\def\level{0.85}     
\def\film{0.38}      

\pgfmathsetmacro{\yTop}{\H}
\pgfmathsetmacro{\yBot}{0}
\pgfmathsetmacro{\yLvl}{\H-\level}

\def\CylinderPath{%
  (-\W/2,\yTop)
    arc[start angle=180,end angle=360,x radius=\W/2,y radius=\E]
  -- (\W/2,\yBot)
    arc[start angle=0,end angle=180,x radius=\W/2,y radius=\E]
  -- cycle
}

\begin{scope}
  \clip \CylinderPath;

  \path[fill=gray!55]
    (-\W/2,\yBot-2*\E) rectangle (-\W/2+\film,\yTop+2*\E);
  \path[fill=gray!55]
    (\W/2-\film,\yBot-2*\E) rectangle (\W/2,\yTop+2*\E);

  \path[fill=gray!15]
    (-\W/2,\yLvl)
      arc[start angle=180,end angle=360,x radius=\W/2,y radius=\E]
    -- (\W/2,\yTop)
      arc[start angle=0,end angle=180,x radius=\W/2,y radius=\E]
    -- cycle;

\end{scope}

\draw (-\W/2,\yBot) -- (-\W/2,\yTop);
\draw ( \W/2,\yBot) -- ( \W/2,\yTop);

\draw (-\W/2,\yBot) arc[start angle=180,end angle=360,x radius=\W/2,y radius=\E];
\draw[dashed] (\W/2,\yBot) arc[start angle=0,end angle=180,x radius=\W/2,y radius=\E];

\draw (-\W/2,\yTop) arc[start angle=180,end angle=360,x radius=\W/2,y radius=\E];
\draw[dashed] (\W/2,\yTop) arc[start angle=0,end angle=180,x radius=\W/2,y radius=\E];

\draw[gray!60] (-\W/2,\yLvl) arc[start angle=180,end angle=360,x radius=\W/2,y radius=\E];
\draw[gray!60,dashed] (\W/2,\yLvl) arc[start angle=0,end angle=180,x radius=\W/2,y radius=\E];


\node[anchor=south] at (-4.4,\yTop+1.45) {\small{fresh medium $S^*$}};
\draw[fill=white, draw=black, line join=miter]
   (-4.7, \yTop+1.35) --        
  (-1.625, \yTop+1.35) --      
  (-1.625, \yLvl+0.05) --
  (-1.425, \yLvl+0.05) --
  (-1.7, \yLvl-0.50) --
  (-1.975, \yLvl+0.05) --
  (-1.775, \yLvl+0.05) --
  (-1.775, \yTop+1.20) --      
  (-4.7, \yTop+1.20) -- cycle; 

\node[anchor=south] at (4.1,\yTop+1.45) {\small{washout ($S, Q$)}};
\draw[fill=white, draw=black, line join=miter]
   (1.375, \yLvl-0.50) --
  (1.375, \yTop+1.20) --       
  (4.3, \yTop+1.20) --         
  (4.3, \yTop+1.00) --         
  (4.7, \yTop+1.275) --        
  (4.3, \yTop+1.55) --         
  (4.3, \yTop+1.35) --         
  (1.225, \yTop+1.35) --       
  (1.225, \yLvl-0.50) -- cycle;
\node[align=center] at (0,2.05) {\small{suspended biomass $Q$} \\ \small{\& free substrate $S$}};

\node[anchor=west,align=left] at (\W/2+1.25,2.35) {\small{biofilm}\\\small{thickness $h$}};
\draw[-Latex] (\W/2+1.05,2.35) -- (\W/2-0.10,2.35);

\end{tikzpicture}
\caption{Schematic view of a CSTR.}
\label{fig:biofilm-reactor}
\end{figure}

The equations reflect that fresh medium is supplied at an inlet substrate concentration $S^*>0$. 
Washout of substrate and suspended biomass is described by a dilution rate $D>0$. 
The constant $k_Q>0$ denotes the death rate of suspended bacteria and  the constant $k_1=1/V\gamma>0$ involves the reactor volume $V$ and the biomass yield per unit of substrate~$\gamma$. 
Furthermore, $\alpha>0$ and $d=d(h)$ represent the attachment rate of suspended cells to the biofilm and the detachment rate from the wall, respectively.
Biomass growth due to substrate consumption is modeled by substrate-dependent growth functions 
$r=r(c)$ for the biofilm population and $\nu=\nu(S)$ for the suspended population. 
The term 
$\rho \partial_z c\big(t,h(t)\big)$ denotes the substrate flux from the aqueous phase into the biofilm, where $\rho>0$ is the biomass density within the biofilm. 
The constant $\beta$ is defined as the product of $\rho$ and the colonizable surface area $A>0$.
Moreover,
\[
\int_0^{h(t)} g\big(r(c(t,z))\big)\,\rd z
\]
describes the growth-induced velocity of the biomass at the moving boundary $z=h(t)$ of the biofilm.  
In~\cite{MasicEberl12}, the substrate-dependent growth functions $r=r(c)$ and $\nu=\nu(S)$ are given as Monod kinetics of the form $r(c)=r_m c/(r_0+c)$ and similarly for $\nu(S$), while the wall detachment rate $d(h)=d_0h$ is linear and the function $g$ is taken as $g(r)=r-b$
with $b>0$ denoting the death rate of biofilm bacteria. Here, we allow for more general functional forms. We refer to~\cite{MasicEberl12,MasicEberl_16} for more details on this particular model and, e.g., to~\cite{BallykJonesSmith,Gaebler21,GaeblerHughesEberl_BMB_21,Jones03,ME14,ME14b,KD10,Stemmons} and the references therein for extended or related models.\\

In~\cite{MasicEberl12}, the well-posedness of~\eqref{C} was established (for slightly less general data), and conditions for the local (in-)stability of the trivial (washout) equilibrium were derived. Moreover, insightful numerical simulations were presented which reveal the existence of a nontrivial equilibrium when the washout equilibrium is unstable and show  persistence of solutions. In addition, structural relationships between model parameters and dependent variables were identified.

The aim of the present work is to provide a more detailed, rigorous mathematical analysis of the model with respect to well-posedness and long-term dynamics. In particular, we provide conditions for the global stability of the washout equilibrium. We also prove the existence of a nontrivial equilibrium within the parameter regime where the washout equilibrium is unstable. Moreover, we provide conditions that guarantee the uniqueness of the nontrivial equilibrium and its local stability. The uniqueness and stability analysis of this equilibrium require careful and nontrivial arguments. It is worth pointing out that even for the algebraically simpler Freter model, the analysis of nontrivial equilibria remains challenging~\cite{Jones03,Stemmons}, and uniqueness can be obtained only under additional assumptions to the best of our knowledge.  Our approach is inspired by~\cite{GW_QAM}, which investigates a related but simpler one-dimensional biofilm model coupling only the biofilm thickness $h$ and the substrate concentration $c$.

In Section~\ref{Sec2}, we establish well-posedness of~\eqref{C} (see Theorem~\ref{T1}) and derive further qualitative properties of solutions needed for the subsequent analysis. In Section~\ref{Sec3}, we first show that the omega-limit set of any (nonnegative) initial value is nonempty (see Proposition~\ref{PX}) and then investigate local (Proposition~\ref{P2}) and global (Corollary~\ref{C10}) stability properties of the trivial (washout) equilibrium. Finally, in Section~\ref{Sec4}, we establish a general existence result for nontrivial equilibria (Theorem~\ref{P5}). While the results obtained so far depend mainly on the magnitude of the inlet $S^*$, we then impose additional structural assumptions on the parameters and employ a shooting argument to prove uniqueness for the nontrivial equilibrium (Theorem~\ref{T5}). Finally, we derive conditions guaranteeing local stability of the (unique) nontrivial equilibrium (Theorem~\ref{T6}). Part of the stability analysis is deferred to Appendix~\ref{App}. In Appendix~\ref{AppN} we include some basic numerical simulations.

\medskip

Throughout the remainder, we set  $$k_2:=D+k_Q>0$$
and assume that
\begin{subequations}\label{GA}
\begin{equation}\label{G0}
g\in {\rm C}^{1-}(\R)\, , 
\end{equation}
where ${\rm C}^{1-}$ refers to locally Lipschitz continuity, that
\begin{equation}\label{G1}r\in
{\rm C}^{1}(\R)\ \text{  is nondecreasing with } sr(s)> 0 \text{ for }\ s\not= 0\,,
\end{equation}
and that
\begin{equation}\label{G2}
\nu,d\in {\rm C}^{1-}(\R) \ \text{ with }\ \nu(s), d(s)>  0\ \text{ for }\ s>0  \ \text{ and }\ \nu(0)=0\,.
\end{equation}
\end{subequations}
In fact, it suffices if the functions are only defined for nonnegative arguments. We shall impose further assumptions when needed.

\section{Well-Posedness}\label{Sec2}

We prove the global well-posedness of system~\eqref{C} and derive additional properties of the solution for later use. As for the well-posedness of \eqref{C}, we temporarily neglect the time dependence and introduce dimensionless variables
\begin{align}\label{dvar}
y=\frac{z}{h}\quad \text{ and }\quad u(y)=c\bigl(yh\bigr),
\end{align}
so that for fixed time,  problem \eqref{C1}-\eqref{C2} for the biomass substrate $c$ on the varying domain $[0,h]$ is equivalent to the
following (fixed) boundary value problem
\begin{subequations}
\label{U}
\begin{align}
\kappa \partial_{y}^2u&= h^2r(u)\,, \quad 0<y<1\,,\quad \label{BBB2}\\
\partial_{y}u(0)&=0\,,\quad u(1)=S\,, \label{BBB4}
\end{align}
\end{subequations}
that we analyze first: 

\begin{prop}\label{P1}
Assume \eqref{GA}. Given $h,S\in\R$, there exists a
unique solution $$u=u[h,S]\in {\rm C}^3 \bigl( [0,1]\bigr)$$ of the
boundary value problem~\eqref{U}.
It holds that
\begin{align}\label{u}
  u(y)=S-\frac{h^2}{\kappa}
  \int_y^1\int_0^\eta r \bigl( u(\xi)\bigr)\,\rd\xi \,\rd \eta,\quad y\in [0,1]\,,
\end{align}
and
$$
\Bigl[ (h,S)\mapsto u[h,S]\Bigr]\in {\rm C}^1\big(\R^2,
{\rm C}^2([0,1])\big)\, .
$$
Moreover, if $S>0$, then 
\begin{subequations}\label{ccc}
\begin{align}
&0< u(0)\le u(y)\le S\,,\quad  y\in [0,1]\,,\label{ccc1}\\
&0\le \partial_y u(y)\le \frac{h^2}{\kappa}r(u(y))\le \frac{h^2}{\kappa}r(S)\,,\quad 
y\in [0,1]\,,\label{ccc2}\\
&0\le \partial_{y}^2u(y)\le \frac{h^2}{\kappa}r(S)\,,\quad  y\in [0,1]\,. \label{ccc3}
\end{align}
\end{subequations}
\end{prop}

\begin{proof}
{\bf (i)} Given $h,S\in\R$,  we define
\begin{equation}\label{F}
F(u)(y):=S-\frac{h^2}{\kappa}
  \int_y^1\int_0^\eta  r \bigl( u(\xi)\bigr)\,\rd\xi \,\rd \eta\,,\quad y\in [0,1]\,,
\end{equation}
for $u\in {\rm C}([0,1])$. The
function $F$ clearly satisfies
$$
F: {\rm C}([0,1])\rightarrow {\rm C}([0,1])
$$
since $ r$ is
continuous. Thanks to
\begin{equation}\label{i97}
\partial_y F(u)(y)=\frac{h^2}{\kappa} \int_0^y  r\bigl(
u(\xi)\bigr)\,\rd \xi,\quad y\in [0,1],
\end{equation}
for  $u\in {\rm C}([0,1])$ and since $r$ is bounded on bounded sets,
the Arzel\`a-Ascoli Theorem ensures that the mapping $F:
{\rm C}([0,1])\rightarrow {\rm C}([0,1])$ is also
compact. Consider $\lambda\in [0,1]$ and $u\in {\rm C}([0,1])$ such that $u=\lambda F(u)$. Then $u\in {\rm C}^2([0,1])$ satisfies
\begin{align*}
\kappa \partial_{y}^2u&= \lambda h^2r(u)\,,  \quad 0<y<1\,,\\
\partial_{y}u(0)&=0\,,\quad u(1)=\lambda S\,.
\end{align*}
For $u_+=\max\{u,0\}$ and $u_-=\min\{u,0\}$ it follows from \eqref{G1} that $u_\pm r(u)\ge 0$ in $[0,1]$ and thus
\begin{align*}
0\le \lambda h^2\int_0^1u_\pm(y) r(u(y))\,\rd y&=-\kappa\int_0^1\vert \partial_y u_\pm(y)\vert^2\,\rd y+\kappa u_\pm(y)\partial_yu(y)\big\vert_{y=0}^{y=1}\\
&=-\kappa\int_0^1\vert \partial_y u_\pm\vert^2\,\rd y+\kappa u_\pm(1)\partial_yu(1)\,.
\end{align*} 
If $S\ge 0$, then $u_-(1)=0$ and thus $\|\partial_y u_-\|_2=0$ so that $u_-\equiv 0$. Hence $u\ge 0$ in $[0,1]$. Since then $r(u)\ge 0$ in $[0,1]$ by \eqref{G1}, we infer from $u=\lambda F(u)$ and \eqref{F} that $0\le u\le S$ in $[0,1]$. Analogously, if $S< 0$, then $u_+(1)=0$ and thus  $u_+\equiv 0$; that is, $u\le 0$ in $[0,1]$. Since then $r(u)\le 0$ in $[0,1]$, we deduce from \eqref{F} that $S\le u\le 0$ in $[0,1]$. Consequently, the set
$$
\left\{u\in C([0,1])\, \big |\,  \text{there is } \lambda\in [0,1] \text { with }
u=\lambda F(u)\right\} 
$$
is bounded. Consequently, Sch\"afer's Fixed Point Theorem \cite[Theorem~11.3]{GT} entails that
there is $u\in {\rm C}([0,1])$ with $u=F(u)$. In fact, 
we also have that $u\in {\rm C}^3([0,1])$ (since $r$ is ${\rm C}^1$) satisfies~\eqref{U}.
As for uniqueness, take two solutions $u_1$ and $u_2$ of
\eqref{U}. Then $w:=u_1-u_2$ satisfies
$$
\kappa\partial_{y}^2w= h^2\big(r(u_1)-r(u_2)\big), \quad y\in [0,1],
$$
and
$$
\partial_y w(0)=0,\quad w(1)=0. 
$$
Multiplying the first equation with $w$ and integrating by parts
gives
\begin{align*}
0&\le h^2 \int_0^1\Big[r\bigl(u_1(y)\bigr)-r\bigl(u_2(y)\bigr)\Big]\Big[u_1(y)-u_2(y)\Big]\,\rd y= -\kappa \int_0^1 \vert\partial_yw(y)\vert^2\,\rd y
\end{align*}
using the boundary conditions and the monotonicity of $r$. It follows
that $w\equiv 0$ on $[0,1]$.\\[.1cm]

{\bf (ii)} To prove the continuous dependence of the
solution ~$u=u[h,S]$ on $(h,S)$, we define
$$
G\in {\rm C}^1\left(\R^2\times {\rm C}^2([0,1]),{\rm C}^2([0,1])\right)
$$ 
by
\begin{equation*}
G(h,S,v)(y):=S-\frac{h^2}{\kappa}
  \int_y^1\int_0^\eta r \bigl( v(\rho)\bigr)\,\rd\rho \,\rd \eta-v(y)\,,\quad y\in [0,1]\,,
\end{equation*}
for $(h,S)\in\R^2$ and $v\in {\rm C}^2([0,1])$. Then, for $(h,S)\in\R^2$ fixed,
$$
G\bigl(h,S,u[h,S]\bigr)=0
$$ 
and its Fr\'echet derivative with respect to $v$ is given by $D_v
G\bigl(h,S,u[h,S]\bigr)=-1-K$, where
\begin{align*}
(K\phi) (y):&=\frac{h^2}{\kappa}
  \int_y^1\int_0^\eta r' \bigl( u[h,S](\xi)\bigr) \phi(\xi)\,\rd\xi \,\rd \eta\,,\quad y\in [0,1]\,,
\end{align*}
for $\phi\in
{\rm C}^2([0,1])$. Since
$$
\partial_y (K\phi) (y)=-\frac{h^2}{\kappa} \int_0^y  r'\bigl(
u[h,S](\xi)\bigr)\phi(\xi)\,\rd \xi\,,\qquad 
\partial_y^2 (K\phi) (y)=-\frac{h^2}{\kappa}  r'\bigl(
u[h,S](y)\bigr)\phi(y)\,,
$$
the Arzel\`a-Ascoli Theorem ensures, as in {\bf (i)}, that $K\in
\mathcal{L}({\rm C}^2([0,1]))$ is a compact operator, hence
$D_v G(h,S,u[h,S]) =-1- K$ is a Fredholm operator. If $\phi\in \mathrm{ker}
\bigl(D_v G\bigl(h,S,u[h,S]\bigr)\bigr)$, then $\phi\in
{\rm C}^2([0,1])$ satisfies
\begin{align*}
\kappa\partial_y^2\phi(y)&= h^2r'\bigl(u[h,S](y)\bigr)\phi(y)\,, \quad y\in [0,1]\,,\\
\partial_y\phi(0)&=0,\quad \phi(1)=0\,.
\end{align*}
Since $r'\bigl(u[h,S](y)\bigr )\ge 0$ for $y\in [0,1]$, it follows as
in the uniqueness proof that $\phi\equiv 0$ on~$[0,1]$. This means
that $D_v G\bigl( h,u[h]\bigr)$ is an automorphism on
${\rm C}^2([0,1])$ for $h,S\in\R^2$. Since $G\bigl(
h,S,u[h,S]\bigr)=0$ for $h,S\in\R^2$, the Implicit Function Theorem now
ensures that indeed
$$
\big[ (h,S)\mapsto u[h,S] \big]\in {\rm C}^1\bigl(\R^2,
{\rm C}^2([0,1])\big)\,,
$$

{\bf (iii)}  Finally, we have already observed that if $S>0$, then $0\le u\le S$ in $[0,1]$. This also implies~\eqref{ccc2} and~\eqref{ccc3}  since $r'\ge 0$.  
\end{proof}

For later use we note:

\begin{cor}\label{Cor1}
Assume \eqref{GA} and let $(h,S)\in [0,\infty)^2$.

{\bf (a)} The function $y\mapsto \partial_S u[h,S](y)$  is nondecreasing with respect to  $y\in [0,1]$ and
\begin{equation} \label{w1}
0<\partial_S u[h,S](0)\le \partial_S u[h,S](y)\le 1\,,\quad y\in [0,1]\,.
\end{equation}
Moreover, it holds that
$$
\partial_S\partial_y u[h,S](y)=\partial_y\partial_S u[h,S](y)\ge 0\,,\quad y\in [0,1]\,.
$$
{\bf (b)} It holds that $\partial_h u[h,S](y)\le 0$ for $y\in[0,1]$. In fact, $\partial_h u[0,S](y)=0$ for $y\in[0,1]$ and in particular $\partial_h\partial_y u[0,S](y)=0$ for $y\in[0,1]$.
\end{cor}

\begin{proof}
Let $u=u[h,S]$.

{\bf (a)} Setting $w(y):=\partial_S u[h,S](y)$ for $y\in [0,1]$, it follows from~\eqref{U} that $w\in {\rm C}^2([0,1])$ satisfies
\begin{align*}
\kappa \partial_{y}^2 w&= h^2r'(u(y))w\,, \quad 0<y<1\,,\\
\partial_{y}w(0)&=0\,,\quad w(1)=1\,. 
\end{align*}
Testing with $w_-$ and using $w_-(1)=0$ and $r'\ge 0$ yields
\begin{align*}
0\le\frac{ h^2}{\kappa}\int_0^1\chi_{[w<0]}w^2(y) r'(u(y))\,\rd y&=-\int_0^1\vert \partial_y w_-(y)\vert^2\,\rd y
\end{align*} 
and hence $w_-\equiv 0$. Therefore, $w\ge 0$ in $[0,1]$ and thus $\partial_y^2 w\ge 0$ since $r'(u)\ge 0$. Since $\partial_y w(0)=0$, we deduce that $w$ is nondecreasing. Assuming $w(0)=0$ for contradiction we had
\begin{align*}
\kappa \partial_{y}^2 w&= h^2r'(u)w\,, \quad 0<y<1\,,\\
w(0)&=\partial_{y}w(0)=0\,, 
\end{align*}
and thus $w\equiv 0$ by uniqueness, which is not possible. This yields~\eqref{w1}. It then remains to note from~\eqref{u} that
\begin{align*}
\partial_S\partial_y u[h,S](y)&=\partial_y w(y)=\frac{h^2}{\kappa}\int_0^y r'\big(u(\xi)\big) w(\xi)\,\rd \xi\ge
\frac{h^2}{\kappa}\big(\min_{[0,S]} r'\big) w(0) y\ge 0
\end{align*}
for $y\in [0,1]$ and the assertion follows.

{\bf (b)} Similarly, setting  $v(y):=\partial_h u[h,S](y)$ for $y\in [0,1]$ we have that $v\in {\rm C}^2([0,1])$ satisfies
\begin{subequations}\label{v}
\begin{align}
\kappa \partial_{y}^2 v&= 2h\, r(u(y))+h^2r'(u(y))v\,, \quad 0<y<1\,,\\
\partial_{y}v(0)&=0\,,\quad v(1)=0\,. 
\end{align}
\end{subequations}
Using $r(u), r'(u)\ge 0$, and $v_+(1)=0$ it follows that
\begin{align*}
0\le\frac{2h}{\kappa}\int_0^1 v_+(y) r(u(y))\,\rd y+ \frac{h^2}{\kappa}\int_0^1\chi_{[v>0]}v^2(y) r'(u(y))\,\rd y&=-\int_0^1\vert \partial_y v_+(y)\vert^2\,\rd y
\end{align*} 
and thus $v_+\equiv 0$. That is, $v\le 0$ in $[0,1]$. Clearly, if $h=0$, then \eqref{v} implies that $\partial_h u[0,S]=v\equiv 0$.
\end{proof}

Using the dimensionless variables~\eqref{dvar} from before and recalling~\eqref{u}, we note that
$$
\partial_z c(h)=\frac{1}{h}\partial_yu(1)=\frac{h}{\kappa} \int_0^1  r\bigl(
u(y)\bigr)\,\rd y\,.
$$
Note that the term on the right-hand side is defined for any $h\in\R$. Therefore, problem~\eqref{C} is (for $h>0$) equivalent to
\begin{subequations}\label{H}
\begin{align}
\kappa \partial_{y}^2u&= h(t)^2 r(u)\,, \qquad 0<y<1\,,  && t>0\,,\label{H1}\\
\partial_{y}u(t,0)&=0,\quad u(t,1)=S(t)\,, &&t>0\,,\label{H2}\\
h'(t)&=h(t)\int_0^{1} g\big(r(u(t,y))\big)\,\rd y +\frac{\alpha}{\beta} Q(t)-d(h(t))h(t)\,,&& t>0\,, \label{H3}\\
S'(t)&=D\big(S^*-S(t)\big)-k_1 Q(t)\nu\big(S(t)\big)-\frac{\rho h(t)}{\kappa} \int_0^1  r\bigl(
u(t,y)\bigr)\,\rd y\,, && t>0\,,\label{H4}\\
Q'(t)&= \big(\nu(S(t))-k_2\big)Q(t)+\beta d\big(h(t)\big)h(t)-\alpha Q(t)\,, && t>0\,,\label{H5}
\end{align}
subject to the initial conditions
\begin{align}
h(0)&=h_0\,,\quad S(0)=S_0\,,\quad Q(0)=Q_0 \label{c6}\,.
\end{align}
\end{subequations}
Given $h,S\in \R$ let  $u[h,S]\in {\rm C}^2([0,1])$ be the unique solution to  \eqref{U} provided by Proposition~\ref{P1}. We write $X:=(h,S,Q)\in\R^3$ and  define
\begin{subequations}\label{f}
\begin{align}
f_1(X)&:=h\int_0^1 g\bigl(r\bigl(u[h,S](y)\bigr)\bigr)\,\rd y +\frac{\alpha}{\beta} Q-d(h)h\,,\\
f_2(X)&:=D(S^*-S)-k_1Q\nu(S)-\frac{\rho h}{\kappa} \int_0^1  r\bigl(
u[h,S](y)\bigr)\,\rd y\,,\\
f_3(X)&:=\big(\nu(S)-k_2\big)Q+\beta d(h)h-\alpha Q\,.
\end{align}
\end{subequations}
Setting $f(X):=(f_1(X),f_2(X),f_3(X))$,  problem~\eqref{H} is equivalent to
\begin{align}\label{CP}
X'=f(X)\,,\quad t>0\,,\qquad X(0)=X_0\,,
\end{align}
This problem has a unique global solution:

\begin{thm}\label{T1}
Assume \eqref{GA}. Then, given $(h_0, S_0, Q_0)\in [0,\infty)^3$, there is a
unique global solution 
$$
(h,S,Q,u)\in {\rm C}^1\big([0,\infty),[0,\infty)^3\times {\rm C}^2([0,1])\big)
$$
to \eqref{H}.
 Moreover, it holds that $$0\le  u(t,y)\le S(t)\le \max\{S^*,S_0\}$$ for $t\ge 0$ and $y\in[0,1]$.
\end{thm}

\begin{proof} 
 Since \mbox{$f\in {\rm C}^{1-}\big(\R^3,\R^3)$} by Proposition~\ref{P1} and Assumption~\eqref{GA}, there exists for each $X_0\in [0,\infty)^3$  a unique maximal solution $X\in {\rm C}^1\bigl([0,T_m),\R^3\bigr)$ to~\eqref{CP}. In fact,  \mbox{$T_m=\infty$} if for every $T>0$ there is $M(T)>0$ such that  $\|X(t)\|_{\R^3}\le M(T)$ whenever $t\in [0,T_m)\cap [0,T]$. Set $u(t,\cdot):=u[h(t),S(t)]$ for $t\in [0,T_m)$. Since 
$$u[h,0]\equiv 0\,,\qquad r(0)=\nu(0)=0\,,\qquad d(h)\ge 0\,,$$ 
for $h,S\ge 0$ due to Proposition~\ref{P1} and Assumption~\eqref{GA}, it follows for $i=1,2,3$ that $f_i(Y)\ge 0$ for $Y\in [0,\infty)^3$ with $Y_i=0$. Therefore, the solution $X$ is nonnegative in each component. 
Note from~\eqref{H4} that
$$
S'(t)\le D\big(S^*-S(t)\big)\,,\quad t\in [0,T_m)\,,\qquad S(0)=S_0\,,
$$
and hence $0\le S(t)\le S_1:=\max\{S_0,S^*\}$ for $t\in [0,T_m)$. Thus, $0\le r(u(t,\cdot))\le r(S_1)$ for $t\in [0,T_m)$ owing to~\eqref{ccc1} and $r'\ge 0$. Consequently, we deduce from~\eqref{H} that
\begin{align}
\frac{\rd}{\rd t}\left(\beta h+Q+\frac{1}{k_1}S\right)(t)&=h(t)\int_0^1 \left[\beta g\bigl(r\bigl(u(t,y)\bigr)\bigr)-\frac{\rho }{k_1\kappa }  r\bigl(u(t,y)\bigr)\right]\,\rd y\nonumber\\
&\qquad+ \frac{D}{k_1}(S^*-S(t))-k_2Q(t)\label{est}\\
&\le \beta g_0 h(t) + \frac{DS^*}{k_1}\nonumber
\end{align}
for $t\in [0,T_m)$ by Proposition~\ref{P1} with 
$$
g_0:=\max_{s\in [0,r(S_1)]} \vert g(s)\vert \,,
$$
where we may assume that $g_0>0$. It readily follows that
$$
0\le \beta h(t)+Q(t)+\frac{1}{k_1}S(t)\le \left(\beta h_0+Q_0+\frac{1}{k_1}S_0+\frac{DS^*}{k_1 g_0}\right)e^{ g_0 t}\,,\quad t\in [0,T_m)\,.
$$
Therefore, $T_m=\infty$.
\end{proof}

\section{Long-Term Dynamics and the Trivial Equilibrium}\label{Sec3}

In the following, we fix $(h_0,S_0,Q_0)\in [0,\infty)^3$ and consider the large-time behavior of the unique global nonnegative solution  $(h,S,Q,u)$ provided by Theorem~\ref{T1}.\\

We begin with a simple observation for the dynamics of $S$:

\begin{lem}\label{C0}
Assume~\eqref{GA}. Either 
\begin{equation}\label{S*}
S(t)\searrow S^* \ \text{ as }\ t\to\infty
\end{equation}
or there is $T_*\in[0,\infty)$ such that
\begin{equation}\label{T*}
S(t)> S^* \ \text{ for }\ 0\le t<T_* \quad \text{ and }\quad S(t)\le S^* \ \text{ for }\ t\ge T_*\,.
\end{equation}
\end{lem}

\begin{proof}
Since
$$
S'(t)\le D\big(S^*-S(t)\big)\,,\quad t\ge 0\,,
$$
it readily follows that $S(t)\le \max\{S(T),S^*\}$ for $t\ge T$ and that $S$ is strictly decreasing as long as $S(t)>S^*$ and that $\limsup_{t\to\infty} S(t)\le S^*$. This implies the assertion.
\end{proof}

\begin{prop}\label{PX}
Assume~\eqref{GA} and that 
\begin{equation}\label{hinfty}
\lim_{z\to\infty}d(z)=\infty\,.
\end{equation}
Then, the orbit $\{(h,S,Q)(t)\,;\, t\ge 0\}$ is bounded in $\R_+^3$. In particular, the omega limit set~$\omega(h_0,S_0,Q_0)$ is nonempty, invariant,  compact, and connected.
\end{prop}

\begin{proof}
We argue similarly as in the proof of Theorem~\ref{T1}. Recall from Lemma~\ref{C0} and~\eqref{GA} that  \mbox{$0\le S(t)\le S_1$} with $S_1=\max\{S_0,S^*\}$ and $0\le r(u(t,\cdot))\le r(S_1)$ for $t\ge 0$. It follows from~\eqref{H} that
\begin{align*}
\frac{\rd}{\rd t}\left(\eta h+Q+\frac{1}{k_1}S\right)(t)&=h(t)\int_0^1 \left[\eta g\bigl(r\bigl(u(t,y)\bigr)\bigr)-\frac{\rho }{k_1\kappa }  r\bigl(u(t,y)\bigr)\right]\,\rd y \\
&\quad -(\eta-\beta) d\big(h(t)\big)h(t)+ \frac{D}{k_1}(S^*-S(t))\\
&\quad -\left(k_2+\alpha\left(1-\frac{\eta}{\beta}\right)\right)Q(t)
\end{align*}
so that, choosing $\eta\in\left(\beta,\frac{k_2+\alpha}{\alpha}\beta\right)$ and setting
$$
\eta_0:=k_2+\alpha\left(1-\frac{\eta}{\beta}\right)>0\,,\qquad \eta_1:=\eta-\beta>0\,,\qquad g_0:=\max_{s\in [0,r(S_1)]} \vert g(s)\vert \,,
$$
we obtain
\begin{align*}
\frac{\rd}{\rd t}\left(\eta h+Q+\frac{1}{k_1}S\right)(t)&\le \big[\eta g_0-\eta_1 d\big(h(t)\big)\big] h(t) + \frac{D}{k_1}(S^*-S(t))-\eta_0 Q(t)
\end{align*}
for $t\ge 0$. Now, due to~\eqref{hinfty} we may choose $h_1>0$ such that 
$$
\eta g_0-\eta_1 d(h)\le -\eta\,,\quad h\ge h_1\,.
$$
Since 
$$
\big(\eta g_0-\eta_1 d(h)\big)h\le \eta(g_0+1)h_1-\eta h\,,\quad 0\le h\le h_1\,,
$$
we set $\delta:=\min\{1,D,\eta_0\}$ and obtain
\begin{align*}
\frac{\rd}{\rd t}\left(\eta h+Q+\frac{1}{k_1}S\right)(t)&\le -\delta \left(\eta h+Q+\frac{1}{k_1}S\right)(t) + ( g_0+1)\eta h_1+\frac{DS^*}{k_1}
\end{align*}
in any case. Consequently, we deduce that
\begin{align*}
\left(\eta h+Q+\frac{1}{k_1}S\right)(t)&\le \max\left\{\eta h_0+Q_0+\frac{1}{k_1}S_0,\frac{1}{\delta}( g_0+1)\eta h_1+\frac{DS^*}{\delta k_1}\right\}
\end{align*}
for $t\ge 0$, hence
the  orbit $\{(h,S,Q)(t)\,;\, t\ge 0\}$ is bounded in $\R_+^3$.
\end{proof}

\subsection{The Wash-Out Equilibrium}\label{Sec3a}

We examine  stability properties of the trivial equilibrium
$$
h=0\,,\quad S=S^*\,,\quad Q=0\,,\quad u\equiv S^*\,.
$$
To apply the Principle of Linearized Stability we assume that $g(r)$ and $\nu$ are differentiable at $S^*$ and that  $d$ is differentiable at $0$. Using the notation from Section~\ref{Sec2} and recalling from Proposition~\ref{P1} that $u[0,S^*]=S^*$, the Jacobian of $f$ at this equilibrium is
\begin{equation*}
\partial f(0,S^*,0)=\left(\begin{matrix}
g(r(S^*))-d(0) & 0 & \alpha/\beta\\[4pt]
-\frac{\rho}{\kappa}r(S^*) & -D & -k_1 \nu(S^*)\\[4pt]
\beta d(0) & 0& \nu(S^*)-k_2-\alpha
\end{matrix}\right)\,.
\end{equation*}
The eigenvalues of $\partial f(0,S^*,0)$ are thus given by
$$
\lambda_0:=-D<0
$$
and
$$
\lambda_\pm:=\frac{1}{2}\left(p+q \pm\sqrt{ (p-q)^2+4\alpha d(0)}\right)
$$
with 
$$
p:=g(r(S^*))-d(0) \,,\qquad q:=\nu(S^*)-k_2-\alpha\,.
$$
Note that $\lambda_\pm\in\R$ with $\lambda_+\geq \lambda_-$.
Hence, the local stability of the equilibrium $(0,S^*,0)$ is determined solely by the sign of $\lambda_+$.  
In fact, we have $\lambda_+<0$
and thus local stability provided that
$$
-pq+\alpha d(0)< 0 \quad \text{ and }\quad p+q<0\,,
$$
while
$\lambda_+ >0$
and thus local instability provided that
$$
-pq+\alpha d(0)>0 \quad \text{ or }\quad p+q>0\,.
$$
Consequently, we obtain:

\begin{prop}\label{P2}
Assume~\eqref{GA} and that $g(r)$ and $\nu$ are differentiable at $S^*$ and that  $d$ is differentiable at $0$. Then, the trivial equilibrium $(h,S,Q)=(0,S^*,0)$ is locally asymptotically stable in $\R^3$ provided that
\begin{align}\label{S1}
d(0)\Big( \nu(S^*)-k_2\Big)<g(r(S^*))\Big(\nu(S^*)-k_2-\alpha\Big)\tag{S1} 
\end{align}
and
\begin{align}\label{S2}
\nu(S^*)-k_2-\alpha<d(0)-g(r(S^*))  \tag{S2}
\end{align}
while it is unstable in $\R^3$ provided that
$$
d(0)\big( \nu(S^*)-k_2\big)> g(r(S^*))\big(\nu(S^*)-k_2-\alpha\big)\quad \text{ or }\quad \nu(S^*)-k_2-\alpha > d(0)-g(r(S^*)) \,.
$$
\end{prop}

\begin{rem}\label{R1}
If $d(0)=0$, then the stability condition can be simplified since \eqref{S1} $\&$ \eqref{S2}  are equivalent to  $\nu(S^*)-k_2-\alpha<0$ and $g(r(S^*))< 0$. In particular, if $\nu$ and $g$ (and $r$) are nondecreasing, the wash-out equilibrium $(0,S^*,0)$ is stable for small $S^*$ and unstable for large $S^*$.
\end{rem}

\begin{proof}
Let $d(0)=0$. Assume \eqref{S1} and \eqref{S2}.
We claim that then $g(r(S^*))< 0$. Assume for contradiction that $g(r(S^*))\ge 0$. 
On the one hand,~\eqref{S1} with $d(0)=0$ implies $$\nu(S^*)-k_2-\alpha<-g(r(S^*))\le 0\,.$$ On the other hand, $(S_1)$ and $g(r(S^*))\ge 0$ imply $0<\nu(S^*)-k_2-\alpha$ and a contradiction. Hence, we do indeed have $g(r(S^*))< 0$. But then $(S_1)$ implies that $\nu(S^*)-k_2-\alpha< 0$.

Conversely, $\nu(S^*)-k_2-\alpha<0$ and $g(r(S^*))< 0$ surely imply \eqref{S1} and \eqref{S2}.
\end{proof}

We next show that the trivial equilibrium is globally asymptotically stable in $[0,\infty)^3$. To this end, we assume \eqref{S1} with
\begin{subequations}\label{Hx}
\begin{equation}\label{g2}
\text{$g$ is nondecreasing}
\end{equation}
and  strengthen \eqref{S2} to
\begin{equation}\label{nu}
\text{$\nu$ is nondecreasing with }\ \nu(S^*)<k_2\,.
\end{equation}
\end{subequations}
Under these slightly stronger assumptions, the trivial equilibrium is globally asymptotically stable (hence there is no other equilibrium):

\begin{cor}\label{C10}
Assume~\eqref{GA}, \eqref{S1},  \eqref{Hx}, and $d(x)\ge d(0)$ for $x\ge 0$.
Then $(0,S^*,0)$ is globally asymptotically stable in $[0,\infty)^3$.
\end{cor}

\begin{proof}
By continuity  and assumptions \eqref{S1} and~\eqref{nu} there exists some $\delta>0$ such that
\begin{align}\label{GS2}
d(0)\big( \nu(S^*+\delta)-k_2\big)<g(r(S^*+\delta))\big(\nu(S^*+\delta)-k_2-\alpha\big)
\end{align}
 and
\begin{align}\label{GS1}
\nu(S^*+\delta)<k_2\,.
\end{align}
Since \eqref{GS2} is equivalent to
\begin{align}\label{GS2b}
0\le \alpha d(0)<\big(d(0)-g(r(S^*+\delta))\big)\big(\alpha+k_2-\nu(S^*+\delta)\big)
\end{align}
we first deduce that $d(0)-g(r(S^*+\delta))>0$ due to~\eqref{GS1} and hence, that
$$
\frac{\beta d(0)}{d(0)-g(r(S^*+\delta))}<\frac{\beta(\alpha+k_2-\nu(S^*+\delta))}{\alpha}\,.
$$
Noticing from~\eqref{GS1} that in addition
$$
\beta<\frac{\beta(\alpha+k_2-\nu(S^*+\delta))}{\alpha}\,,
$$
we may then choose $\eta$ such that
\begin{align}\label{etaGS}
\max\left\{\beta,\frac{\beta d(0)}{d(0)-g(r(S^*+\delta))}\right\}<\eta<\frac{\beta(\alpha+k_2-\nu(S^*+\delta))}{\alpha}\,.
\end{align}
Note from Lemma~\ref{C0} that there exists $T_\delta>0$ such that $S(t)\le S^*+\delta$ for $t\ge T_\delta$. Therefore, we infer from~\eqref{g2} that 
\begin{align*}
g\bigl(r\bigl(u(t,y)\bigr)\bigr)\le g(r(S(t)))\le g(r(S^*+\delta))
\qquad \text{and}\qquad  \nu(S(t))\le \nu(S^*+\delta)
\end{align*}
for $t\ge T_\delta$ and $y\in [0,1]$. It then follows from~\eqref{H} that
\begin{align*}
\frac{\rd}{\rd t}\left(\eta h+Q\right)(t)&=h(t)\eta\int_0^1  g\bigl(r\bigl(u(t,y)\bigr)\bigr)\,\rd y-(\eta-\beta)d\big(h(t)\big)h(t)\\
&\quad +\left(\nu(S(t))-k_2-\alpha+\frac{\alpha\eta}{\beta}\right)Q(t)\\
&\le \Big(\eta g(r(S^*+\delta))-(\eta-\beta)d(0)\Big)h(t)\\
&\quad +\left(\nu(S^*+\delta)-k_2-\alpha+\frac{\alpha\eta}{\beta}\right)Q(t)\\
&\le -m \big(\eta h+Q\big)(t)
\end{align*}
for $t\ge T_\delta$ with 
$$
m:=\min\left\{ (1-\frac{\beta}{\eta})d(0)-g(r(S^*+\delta)), \alpha-\frac{\alpha\eta}{\beta}+k_2-\nu(S^*+\delta)\right\}>0\,.
$$
Consequently,
$$
\lim_{t\to\infty} h(t)=0\,,\qquad \lim_{t\to\infty} Q(t)=0\,.
$$
In particular, given $\ve>0$ there is $T_\ve>0$ such that
$$
0\le k_1Q(t)\nu(S(t))+\frac{\rho h(t)}{\kappa} \int_0^1  r\bigl(
u(t,y)\bigr)\,\rd y\le D\ve\,,\qquad t\ge T_\ve\,,
$$ 
so that~\eqref{H4} implies that
$$
S'(t)\ge D\big(S^*-S(t)\big)- D\ve\,,\qquad t\ge T_\ve\,.
$$
Integrating this inequality yields
$$
\liminf_{t\to\infty} S(t)\ge S^*-\ve
$$
and since $\ve>0$ was arbitrary and $\limsup_{t\to\infty} S(t)\le S^*$ by Lemma~\ref{C0}, we deduce that $\lim_{t\to\infty} S(t)= S^*$.
\end{proof}
We refer to Appendix~\ref{AppN} for illustrations of the results this section.

\section{Nontrivial Equilibrium}\label{Sec4}
In this section, we investigate the existence, uniqueness, and stability properties of nontrivial equilibria to~\eqref{H}. A nontrivial equilibrium is a triplet $(h,S,Q)\in (0,\infty)^3$ that, with the corresponding $u=u[h,S]\in C^2([0,1])$ from Proposition~\ref{P1}, satisfies the system 
\begin{subequations}\label{HSS}
\begin{align}
\kappa \partial_{y}^2u&= h^2 r(u)\,, \qquad 0<y<1\,,  \label{HSS1}\\
\partial_{y}u(0)&=0,\quad u(1)=S \,,\label{HSS2}\\
0&=h\int_0^{1} g\big(r(u(y))\big)\,\rd y +\frac{\alpha}{\beta} Q -d(h)h\,, \label{HSS3}\\
0&=D(S^*-S)-k_1 Q\nu(S)-\frac{\rho h}{\kappa} \int_0^1  r\bigl(
u(y)\bigr)\,\rd y\,, \label{HSS4}\\
0&= \big(\nu(S)-k_2\big)Q+\beta d(h)h-\alpha Q\,.\label{HSS5}
\end{align}
\end{subequations}
To begin with, we note that
several necessary conditions for the existence of a nontrivial equilibrium can be derived straightforwardly:

\begin{lem}
Let $(h,S,Q)\in (0,\infty)^3$ be a nontrivial equilibrium. Then necessarily $S<S^*$ and 
\begin{equation}\label{DeltaS}
\alpha+k_2-\nu(S)>0
\end{equation}
with
\begin{equation}\label{Q}
Q=\frac{\beta d(h)h}{\alpha+k_2-\nu(S)}
\end{equation}
and
\begin{equation}\label{gp}
\int_0^{1} g\big(r(u(y))\big)\,\rd y=\frac{\big(k_2-\nu(S)\big) d(h)}{\alpha+k_2-\nu(S)} \,.
\end{equation}
Moreover,  if $g$, $\nu$, and $d$ are nondecreasing, then necessarily
\begin{align}\label{n20}
g(r(S^*))\big(\alpha+k_2-\nu(S^*)\big)\ge d(0)\big(k_2 -\nu(S^*)\big) \quad \text{ or }\quad \nu(S^*)\ge k_2\,.
\end{align}
\end{lem}

\begin{proof}
It readily follows from \eqref{HSS4} that $S<S^*$ while~\eqref{HSS5} yields~\eqref{Q}. The latter yields~\eqref{DeltaS}  and, together with~\eqref{HSS3}, that~\eqref{gp} holds true. Finally, if $g$, $\nu$, and $d$ are nondecreasing, then Corollary~\ref{C10} implies that there is no nontrivial equilibrium if 
$$
g(r(S^*))\big(\alpha+k_2-\nu(S^*)\big)< d(0)\big(k_2 -\nu(S^*)\big) \quad \text{ and }\quad \nu(S^*)< k_2\,.
$$
This gives~\eqref{n20}.
\end{proof}

\subsection{A General Existence Result}

We now establish the existence of at least one nontrivial equilibrium when the stability condition~\eqref{S1} from Proposition~\ref{P2} is violated:

\begin{thm}\label{P5}
Assume~\eqref{GA} with  $g(0)<0$. Let $\nu\in {\rm C}^1([0,S^*])$ be increasing, and let $d\in {\rm C}^1(\R^+)$ with $\lim_{h\to\infty} d(h)h=\infty$. Moreover, assume that
\begin{equation}\label{n1}
\alpha+k_2> \nu(S^*)
\end{equation}
and that
\begin{align}\label{n2}
g(r(S^*))\big(\alpha+k_2-\nu(S^*)\big)> d(0)\big(k_2 -\nu(S^*)\big)\,.
\end{align}
Then, there is at least one nontrivial equilibrium 
$$
(h_\star,S_\star,Q_\star)\in (0,\infty)\times (0,S^*)\times (0,\infty)\,,\qquad  u_\star=u[h_\star,S_\star]\in {\rm C}^2([0,1])\big)
$$ 
to \eqref{H}. 
\end{thm}

\begin{proof} Let us first note that \eqref{n1} and the monotonicity of $\nu$ imply that
\begin{equation}\label{n3}
\alpha+k_2- \nu(S)>0\,,\quad S\in [0,S^*]\,,
\end{equation}
and recall from Proposition~\ref{P1} that
$$
\Bigl[ (h,S)\mapsto u[h,S]\Bigr]\in {\rm C}^1\big(\R^2,{\rm C}^2([0,1])\big)\,,
$$
where $u[h,S]$ is the solution to~\eqref{HSS1}-\eqref{HSS2}.
Next, owing to~\eqref{Q}
we can eliminate $Q$ and obtain from \eqref{gp} the identity
\begin{subequations}\label{Aw}
\begin{equation}\label{A1w}
G(h,S):=\int_0^{1} g\big(r(u[h,S](y))\big)\,\rd y-d(h)\left(1-\frac{\alpha}{\alpha+k_2-\nu(S)}\right)\stackrel{!}{=}0
\end{equation}
while \eqref{HSS4} yields 
\begin{equation}\label{A2}
E(h,S):=DS+\frac{\rho h}{\kappa} \int_0^1  r\bigl(u[h,S](y)\bigr)\,\rd y + \frac{k_1\beta \nu(S) d(h)h}{\alpha+k_2-\nu(S)}\stackrel{!}{=}DS^*\,.
\end{equation}
\end{subequations}
That is, \eqref{HSS} is equivalent to solving~\eqref{Aw} for $(h,S)$ and then defining $Q$ by~\eqref{Q}.

In order to solve \eqref{Aw}, let $h\ge 0$ first be fixed. Using $u[h,0]\equiv 0$, $\nu(0)=0$, $r(0)=0$, and positivity, it readily follows that
\begin{equation}\label{n6}
E(h,0)=0\,,\qquad E(h,S^*)\ge  DS^*\,.
\end{equation}
Since 
\begin{align*}
\partial_SE(h,S)&=D+\frac{\rho h}{\kappa} \int_0^1  r'\bigl(u[h,S](y)\bigr) \partial_Su[h,S](y)\,\rd y+\frac{k_1\beta \nu'(S) d(h)h(\alpha+k_2)}{(\alpha+k_2-\nu(S))^2}\,,
\end{align*}
it follows from Corollary~\ref{Cor1}~{\bf (a)}, and the monotonicity of $r$ and $\nu$ that $E(h,S)$ is strictly increasing with respect to $S$. Therefore, using~\eqref{n6} we find for each $h\ge 0$  a unique $S(h)\in (0,S^*]$ such that
\begin{equation}\label{E}
E\big(h,S(h)\big)=DS^*\,,\quad h\ge 0\,,
\end{equation}
with $S(0)=S^*$. In fact, the Implicit Function Theorem implies that 
$S\in {\rm C}^1(\R^+)$. Thus, setting $\mathsf{G}(h):=G(h,S(h))$ for $h\ge 0$, we obtain $\mathsf{G}\in {\rm C}^1(\R^+)$. Since $S(0)=S^*$ and $u[0,S^*]=S^*$, it follows from assumption~\eqref{n2} that
\begin{equation}\label{n7}
\mathsf{G}(0)=  g\big(r(S^*)\big)-d(0)\left(1-\frac{\alpha}{\alpha+k_2-\nu(S^*)}\right)>0\,.
\end{equation}
Next, the positivity of $r$ and $\nu$ yield that
$$
DS^*=E(h,S(h))\ge  \frac{k_1\beta \nu(S(h)) d(h)h}{\alpha+k_2}
$$
and since $\lim_{h\to\infty} d(h)h=\infty$, we deduce that 
\begin{equation}\label{t1}
\lim_{h\to\infty} S(h)=0\,.
\end{equation} 
Since 
$$
0\le u[h,S(h)](y)\le S(h) \rightarrow 0 
$$
for $y\in [0,1]$ as  $h\rightarrow \infty$
according to Proposition~\ref{P1}, we  conclude that
\begin{equation}\label{n4}
\lim_{h\to\infty} u[h,S(h)](y)=0\ \text{ uniformly with respect to $y\in [0,1]$}\,.
\end{equation}
 Therefore, it follows from~\eqref{n4}, \eqref{GA}, and the assumption $g(0)<0$ that there is $h_0>0$ such that $g(r(u[h_0,S(h_0)](y)))<0$ for $y\in [0,1]$. Moreover, due to~\eqref{t1} and $\nu(0)=0$ we may assume that
$$
1-\frac{\alpha}{\alpha+k_2-\nu(S(h_0))}>0\,.
$$
Hence, we obtain that
\begin{equation}\label{n8}
\mathsf{G}(h_0)\le \int_0^{1} g\big(r(u[h_0,S(h_0)](y))\big)\,\rd y <0\,.
\end{equation}
It thus follows from \eqref{n7}, \eqref{n8}, and the Mean Value Theorem that there exists some $h_\star\in (0,h_0)$ such that $\mathsf{G}(h_\star)=0$. Setting $S_\star:=S(h_\star)\in (0,S^*)$ and defining $Q_\star>0$  and $u_\star\in {\rm C}^2([0,1])$ by~\eqref{Q} and Proposition~\ref{P1}, respectively, we have established the existence of a nontrivial equilibrium for \eqref{H}.
\end{proof}

\begin{rem}
Unfortunately, the proof above does not establish the uniqueness of the nontrivial equilibrium. Indeed,
\begin{align*}
\frac{\rd }{\rd h}\mathsf{G}(h)&=\int_0^{1} (g\circ r)'\big(u[h,S(h)](y)\big) \left\{\partial_h u[h,S(h)](y)+\partial_S u[h,S(h)](y) S'(h)\right\}\,\rd y\\
&\qquad-d'(h)\left(1-\frac{\alpha}{\alpha+k_2-\nu(S(h))}\right) +d(h)\frac{\alpha \nu'(S(h)) S'(h)}{(\alpha+k_2-\nu(S(h)))^2}
\end{align*}
does not necessarily have a definite (negative) sign, even if $g$ and $d$ are nondecreasing, since $S'$ need not have a definite sign. Differentiating \eqref{E} does not determine the sign of $S'$ (recall from Proposition~\ref{P1} that $\partial_h u[h,S]\le 0$ while $\partial_S u[h,S]\ge 0$).
\end{rem}

\subsection{Uniqueness by a Shooting Argument}\label{SS4.2} 

Theorem~\ref{P5} establishes the existence of a nontrivial equilibrium. Under more restrictive assumptions we shall now  derive uniqueness of a nontrivial equilibrium relying on a shooting argument as in~\cite{GW_QAM}. More precisely, we assume that $g$ has the special form
\begin{align}\label{gg}
g(s):=a(s-b)\,,\quad s\ge 0\,,
\end{align}
for some $a,b>0$ as is the case in \cite{MasicEberl12}. In this setting, for an equilibrium to~\eqref{C} we have from~\eqref{C1} and~\eqref{C2} that
\begin{align}\label{gw}
\int_0^h g\big(r(c(z))\big)\,\rd z=a\int_0^h r(c(z))\,\rd z
-a bh=a \kappa \partial_z c(h)- abh\,.
\end{align}
Combining \eqref{gw} and \eqref{C3} we obtain at equilibrium the problem
\begin{subequations}\label{A}
\begin{align}
&\kappa\partial_z^2 c\stackrel{!}{=}r(c)\,, \quad 0<z<h\,, \label{A1}\\
&\partial_z c(0)\stackrel{!}{=}0\,,\quad
 c(h)\stackrel{!}{=}S\,,\quad
a\kappa \partial_z c(h)\stackrel{!}{=}d(h)h+abh-\frac{\alpha}{\beta}Q \label{A4}
\end{align}
for $c$ and $h$, where $Q$ and $S$ we shall derive from the conditions
\begin{equation}\label{A5}
\big(\alpha+k_2-\nu(S)\big)Q\stackrel{!}{=}\beta d(h)h
\end{equation}
respectively
\begin{equation}\label{A6}
DS+k_1Q\nu(S)+\rho\partial_zc(h)\stackrel{!}{=}D S^*\,.
\end{equation}
\end{subequations}
Clearly, \eqref{A5} implies that $Q$ is given by
\begin{equation}\label{A5x}
Q=\frac{\beta d(h)h}{\alpha+k_2-\nu(S)}
\end{equation}
and thus depends only on $h$ and $S$. Plugging this into \eqref{A6} (using~\eqref{A4}), we also obtain a relation only depending on $h$ and $S$:
\begin{equation}\label{A6x}
DS+\frac{\rho b h}{\kappa}+\frac{d(h)h}{\alpha+k_2-\nu(S)} \left[\left(k_1\beta-\frac{\rho}{a\kappa}\right)\nu(S)+\frac{\rho k_2}{a\kappa}\right]\stackrel{!}{=}D S^*\,.
\end{equation}
Consequently, the existence of a unique nontrivial equilibrium to~\eqref{C} is equivalent to find a unique nontrivial solution to~\eqref{A1}-\eqref{A4} and \eqref{A5x}-\eqref{A6x}. In order to do so, we first show that for any fixed \(h\) there exists a unique solution
\(S = S(h)\) to~\eqref{A6x}. Together with~\eqref{A5x}, this allows us to reduce
\eqref{A1}--\eqref{A4} to a problem for the unknowns \(h\) and \(c\) only, which we then
solve by a shooting argument.

To carry through the technical details, we impose that
\begin{subequations}\label{NT}
\begin{equation}\label{h1}
\nu\in {\rm C}^1([0,S^*]) \ \text{is nondecreasing} \,,\quad   d\in {\rm C}^1(\R^+)  \ \text{is strictly increasing}\,,
\end{equation}
and that
\begin{align}\label{h2}
k_2>\nu(S^*)\quad\text{ and }\quad g(r(S^*))\big(\alpha+k_2-\nu(S^*)\big)>d(0)\big(k_2-\nu(S^*)\big)\,.
\end{align}
Note that \eqref{h1}-\eqref{h2} imply in particular the instability of the wash-out equilibrium according to~Proposition~\ref{P2} and that
$$
\alpha+k_2-\nu(S)\ge \alpha+k_2-\nu(S^*)>0\,,\quad 0\le S\le S^*\,.
$$
We further impose that
\begin{align}
 & \frac{\alpha}{a(\alpha+k_2)}\le \frac{\beta k_1 \kappa}{\rho} \label{h3a}\\
\textit{ or }\quad \nonumber\\
 &\text{$S\mapsto \frac{\nu'(S)}{(\alpha+k_2-\nu(S))^2}$ is strictly monotone on $[0,S^*]$}\,.\label{h3b}
\end{align}
\end{subequations}

Under these assumptions we can prove the existence of a {\it unique} nontrivial equilibrium:

\begin{thm}\label{T5}
Assume~\eqref{GA}, \eqref{gg}, and~\eqref{NT}.
Then, there exists a unique nontrivial equilibrium solution 
$$
(h_\star,S_\star,Q_\star)\in (0,\infty)\times (0,S^*)\times (0,\infty)\,,\qquad  c_\star\in {\rm C}^2([0,h_\star])\big)
$$ 
to \eqref{C}. 
\end{thm}

Note that conditions~\eqref{h3a} and~\eqref{h3b} both constitute structural assumptions that are independent of the magnitude of the inlet concentration $S^*$ and are not needed for the instability of the trivial equilibrium (they are used in the proof of Proposition~\ref{L20} below).
In this regard it is worth to note the following:

 \begin{rem}\label{R10}
 In  \cite{MasicEberl12}, the function $\nu$ is of the form $\nu(S)=AS/(B+S)$ with given \mbox{$A,B>0$}. It readily follows that such a function satisfies assumption~\eqref{h3b}  provided that $\alpha+k_2\not=A$. Moreover, the functions $r$, $d$, and  $g$ in \cite{MasicEberl12} satisfy~\eqref{GA},~\eqref{gg}, and~\eqref{h1}. Therefore, for the particular case considered in \cite{MasicEberl12}, we prove herein the existence of a unique nontrivial equilibrium under the sole assumptions \eqref{h2} and $\alpha+k_2\not=A$. 
\end{rem}

The proof of Theorem~\ref{T5} is divided into several steps.
Focusing first on~\eqref{A6x} we express $S$ as a function of $h$. For this purpose we set
$$
F(h,S):=DS+\frac{\rho b h}{\kappa}+\frac{d(h)h}{\alpha+k_2-\nu(S)} \left[\left(k_1\beta-\frac{\rho}{a\kappa}\right)\nu(S)+\frac{\rho k_2}{a\kappa}\right]
$$
for $h\ge 0$ and $S\in [0,S^*]$.

\begin{prop}\label{L20}
Assume~\eqref{GA}, \eqref{gg}, and~\eqref{NT}.
There exist a unique $h_*>0$ and a unique, strictly decreasing function $S\in {\rm C}^1([0,h_*], [0,S^*])$ with $S(0)=S^*$ and $S(h_*)=0$ such that 
$$
F\big(h,S(h)\big)=DS^*\,,\quad h\in [0,h_*]\,.
$$
\end{prop}

\begin{proof}
Since $\nu(0)=0$, we have
$$
F(h,0)=\frac{\rho h}{a\kappa}\left(ab+\frac{k_2d(h)}{\alpha+k_2}  \right)\,,\quad h>0\,.
$$
 Consequently, there is a unique $h_*>0$ such that 
\begin{align}\label{c12}
F(h,0)<DS^*\,,\quad 0\le h< h_*\,,\qquad F(h_*,0)=DS^*\,.
\end{align} 
Next, we derive from \eqref{h1} and~\eqref{h2} that
\begin{align}\label{c13}
F(h,S^*)=DS^* +\frac{k_1\beta\nu(S^*)d(h)h}{\alpha+k_2-\nu(S^*)}+\frac{\rho h}{a\kappa}\left(ab+\frac{k_2-\nu(S^*)}{\alpha+k_2-\nu(S^*)}d(h)\right)> DS^*
\end{align}
for $0< h\le h_*$ and $F(0,S^*)=DS^*$. Since
\begin{align}\label{c14}
\partial_SF(h,S)=D+\frac{\nu'(S)d(h)h}{(\alpha+k_2-\nu(S))^2} \left[k_1\beta(\alpha+k_2)-\frac{\rho\alpha}{a\kappa}\right]\,,
\end{align}
alternative~\eqref{h3a} implies that $\partial_SF(h,S)>0$ while  alternative~\eqref{h3b} implies that $\partial_SF(h,S)$  is strictly monotone in $S$ (and hence $F(h,\cdot)$ is strictly convex or concave). In any case, for each $h\in [0,h_*]$, the function $S\mapsto \partial_SF(h,S)$ has at most one zero. Therefore, since $F(h,0)<DS^*<F(h,S^*)$ for $h\in (0,h_*]$ as observed above, it follows that for each $h\in [0,h_*]$ there is a unique $S(h)\in [0,S^*]$ such that $F(h,S(h))=DS^*$. 
Moreover, we have $S(0)=S^*$ and $S(h_*)=0$ by construction. In addition, since $\partial_SF(h,S)>0$ or $F(h,\cdot)$ is strictly convex or concave as just proven, we necessarily have 
\begin{align}\label{c15}
\partial_S F(h,S(h))>0\,,\quad h\in (0,h_*)\,.
\end{align} 
The Implicit Function Theorem and \eqref{c15} then imply that $S\in C^1([0,h_*])$. It remains to show that~$S'\le 0$.
To this end, note  that
\begin{align*}
\partial_h F(h,S)&=\frac{d'(h)h}{\alpha+k_2-\nu(S)} \left[ k_1\beta\nu(S)+\frac{\rho}{a\kappa}(k_2-\nu(S))\right]\\
&\qquad +\frac{k_1\beta \nu(S)d(h)}{\alpha+k_2-\nu(S)} +\frac{\rho}{a\kappa}\left( ab+\frac{k_2-\nu(S)}{\alpha+k_2-\nu(S)}d(h)\right) \,.
\end{align*}
We infer from~\eqref{h2} and the monotonicity of $d$ that $\partial_h F(h,S(h))>0$ for $h\in [0,h_*]$. Differentiating $F(h,S(h))=DS^*$, we obtain
$$
\partial_h F(h,S(h))+S'(h)\partial_SF(h,S(h))=0
$$
for $h\in (0,h_*)$, hence $S'(h) <0$ for $h\in [0,h_*]$ due to~\eqref{c15}. This proves Proposition~\ref{L20}.
\end{proof}

It remains to consider~\eqref{A1}-\eqref{A4}, where $S=S(h)$ is given by Proposition~\ref{L20} and \mbox{$Q=Q(h)$} as in~\eqref{A5x}, that is, to find a unique $h\in (0,h_*)$ and a unique function \mbox{$c \in {\rm C}^2([0,h])$} such that
\begin{subequations}\label{eq}
\begin{align}
&\kappa\partial_z^2 c\stackrel{!}{=}r(c)\,, \quad 0<z<h\,, \label{A1s}\\
&\partial_z c(0)\stackrel{!}{=}0\,,\quad
 c(h)\stackrel{!}{=}S(h)\,,\quad
a\kappa \partial_z c(h)\stackrel{!}{=}abh+\frac{(k_2-\nu(S(h)))d(h)h}{\alpha+k_2-\nu(S(h))}\,. \label{A4s}
\end{align}
\end{subequations}
To this end, we use a shooting argument. The following result is readily seen, we refer to~\cite[Lemma~3.7]{GW_QAM}:

\begin{lem}\label{L17} 
For each $\mu\in [0,S^*]$ there exists a unique function ${\rm c}(\cdot,\mu)\in {\rm C}^2([0,\infty))$ satisfying
\begin{subequations}\label{shooting}
\begin{align}
&\kappa\partial_z^2 {\rm c}=r({\rm c})\,, \quad z>0\,, \label{A1ss}\\
&{\rm c}(0)=\mu\,,\quad \partial_z {\rm c}(0)=0\,. \label{A4ss}
\end{align}
\end{subequations}
The functions $z\mapsto {\rm c}(z,\mu)$ and $z\mapsto \partial_z{\rm c}(z,\mu)$ are increasing. In addition,
$$
\big[\mu\mapsto c(\cdot,\mu)\big]\in {\rm C}^1\big( [0,S^*],{\rm C}^2([0,R])\big)
$$
for each $R>0$. Moreover, $w(z,\mu):=\partial_\mu {\rm c}(z,\mu)$ satisfies
\begin{subequations}\label{shootingm}
\begin{align}
&\kappa\partial_z^2 w(z,\mu)=r'({\rm c}(z,\mu))w(z,\mu)\,, \quad z>0\,, \label{A1ssm}\\
&w(0,\mu)=1\,,\quad \partial_z w(0,\mu)=0\,, \label{A4ssm}
\end{align}
\end{subequations}
with $w(z,\mu)\ge 1$ and $\partial_zw(z,\mu)\ge 0$ for $z\ge 0$ and  $\mu\in [0,S^*]$.
\end{lem}

Now, the aim is to show that there exists a unique $\mu\in [0,S^*]$ such that ${\rm c}(\cdot,\mu)$ solves~\eqref{eq} for a unique $h(\mu)>0$.

We first comply with the Dirichlet boundary condition at $z=h$ in \eqref{A4s}:

\begin{prop}\label{L21}
Assume~\eqref{GA}, \eqref{gg}, and~\eqref{NT}.
Let ${\rm c}(\cdot,\mu)\in {\rm C}^2([0,\infty))$  be the unique solution to the shooting problem~\eqref{shooting}.
Then, there exists a unique, strictly decreasing function $h\in {\rm C}^1([0,S^*])$  with $h(0)=h_*$ and $h(S^*)=0$ such that 
\begin{align}\label{id}
{\rm c}\big(h(\mu),\mu\big)=S\big(h(\mu)\big)\,,\quad \mu\in [0,S^*]\,.
\end{align}
\end{prop}

\begin{proof}
Define
$$
A(z,\mu):={\rm c}(z,\mu)-S(z)\,,\qquad z\in [0,h_*]\,,\quad \mu\in [0,S^*]\,.
$$
Then, for fixed $\mu\in [0,S^*)$,
$$
A(0,\mu)={\rm c}(0,\mu)-S(0)=\mu-S^*<0\,,\qquad A(h_*,\mu)={\rm c}(h_*,\mu)-S(h_*)={\rm c}(h_*,\mu)>0
$$
by Proposition~\ref{L20}. Moreover, 
$$
\partial_z A(z,\mu)=\partial_z{\rm c}(z,\mu)-S'(z)>0\,,\quad z\in (0,h_*]\,,
$$
since \(z\mapsto {\rm c}(z,\mu)\) is increasing by Lemma~\ref{L17} and \(S'(z)<0\) by Proposition~\ref{L20}. Therefore, for each $\mu\in [0,S^*]$ there is a unique $h(\mu)\in [0,h_*]$ such that $A(h(\mu),\mu)=0$ (with $h(0)=h_*$ and $h(S^*)=0$), and the Implicit Function Theorem yields that $h\in {\rm C}^1([0,S^*])$. Noticing that
\begin{align*}
0&=\partial_zA(h(\mu),\mu) h'(\mu)+\partial_\mu A\big(h(\mu),\mu\big)=\partial_zA\big(h(\mu),\mu\big) h'(\mu)+\partial_\mu {\rm c}\big(h(\mu),\mu\big)
\end{align*}
with $\partial_zA\big(h(\mu),\mu\big)>0$ and $\partial_\mu {\rm c}\big(h(\mu),\mu\big)\ge 1$, we derive that $h'(\mu)<0$ for $\mu\in [0,S^*)$.
\end{proof}

In order to comply with the Neumann condition at $z=h$ we will require the following auxiliary result:

\begin{lem}\label{L22}
Assume~\eqref{GA}, \eqref{gg} and~\eqref{NT}. Set $w(z,\mu):=\partial_\mu{\rm c}(z,\mu)$ and define
\begin{align*}
M(z,\mu)&:=-a\big[r({\rm c}(z,\mu))-b\big]w(z,\mu)+a\kappa\big(\partial_z{\rm c}(z,\mu)-S'(h(\mu))\big) \partial_z w(z,\mu)\\
&\qquad +\frac{k_2-\nu\big(S(h(\mu))\big)}{\alpha+k_2-\nu\big(S(h(\mu))\big)}d\big(h(\mu)\big)
\end{align*}
for $(z,\mu)\in [0,h_*]\times [0,S^*]$. Then, for fixed $\mu\in [0,S^*]$, the function
$$
M(\cdot,\mu):[0,h_*]\to \R
$$ 
is increasing. Moreover,
$M(0,\cdot):[0,S^*]\to\R$ with
$$
M(0,\mu)=-a\big(r(\mu)-b\big)+\frac{k_2-\nu\big(S(h(\mu))\big)}{\alpha+k_2-\nu\big(S(h(\mu))\big)}d\big(h(\mu)\big)\,,\quad \mu\in [0,S^*]\,,
$$
is strictly decreasing, and there exists (a unique) $\underline{\mu}\in (0,S^*)$ such that $M(0,\underline{\mu})=0$. In particular,
\begin{equation}\label{M}
M(z,\mu)\ge M(0,\mu)>0\,,\qquad z\in [0,h_*]\,,\quad \mu\in [0,\underline{\mu})\,.
\end{equation}
\end{lem}

\begin{proof}
It readily follows from Lemma~\ref{L17}  that
\begin{align*}
\partial_z M(z,\mu)&=ab\partial_z w(z,\mu)-aS'\big(h(\mu)\big)r'\big({\rm c}(z,\mu)\big)w(z,\mu)\ge 0
\end{align*}
for $(z,\mu)\in [0,h_*]\times [0,S^*]$, where we used the monotonicity properties of $w$, $r$, and $S$ (see~\eqref{GA}, Proposition~\ref{L20}, and Lemma~\ref{L17}). This ensures that \mbox{$M(\cdot,\mu):[0,h_*]\to \R$} is increasing.

Next, recalling that $w(0,\mu)=1$ and $\partial_zw(0,\mu)=0$, we have
$$
M(0,\mu)=-a\big(r(\mu)-b\big)+\frac{k_2-\nu\big(S(h(\mu))\big)}{\alpha+k_2-\nu\big(S(h(\mu))\big)}d\big(h(\mu)\big)\,,\quad \mu\in [0,S^*]\,.
$$
Since $r(0)=0$ by~\eqref{GA}, $\nu(0)=0$ by~\eqref{h1}, $S(h_*)=0$ and $h(0)=h_*$ by Proposition~\ref{L20} and, respectively, Proposition~\ref{L21}, we obtain that
$$
M(0,0)=ab+\frac{k_2}{\alpha+k_2}d(h_*)>0
$$
and, recalling $h(S^*)=0$ and $S(0)=S^*$ and using \eqref{gg} and~\eqref{h2}, that
$$
M(0,S^*)=-a\big(r(S^*)-b\big)+\frac{k_2-\nu(S^*)}{\alpha+k_2-\nu(S^*)}d(0)<0\,.
$$
Thus, the continuity of $M(0,\cdot)$ implies the existence of $\underline{\mu}\in (0,S^*)$ such that $M(0,\underline{\mu})=0$. In fact, since
\begin{align*}
\frac{\rd}{\rd \mu}M(0,\mu)&=-a r'(\mu)-\alpha\frac{d\big(h(\mu)\big)}{\big(\alpha+k_2-\nu\big(S(h(\mu))\big)\big)^2} \frac{\rd}{\rd \mu}\nu\big(S(h(\mu))\big)\\
&\qquad+\frac{k_2-\nu\big(S(h(\mu))\big)}{\alpha+k_2-\nu\big(S(h(\mu))\big)}\frac{\rd}{\rd \mu}d\big(h(\mu)\big)
\end{align*}
and since $r$ and $\nu$ are  nondecreasing, $d$ is strictly increasing, while $S$ and $h$ are strictly decreasing, we deduce from assumption~\eqref{h2} that~$M(0,\cdot)$ is strictly decreasing on $[0,S_*]$. This implies the assertion.
\end{proof}

Guided by~\eqref{A4s}, we now introduce
$$
B(\mu):=a\kappa \partial_z {\rm c}\big(h(\mu),\mu\big)-abh(\mu)-\frac{k_2-\nu\big(S(h(\mu))\big)}{\alpha+k_2-\nu\big(S(h(\mu))\big)}d\big(h(\mu)\big)h(\mu)\,,\quad \mu\in [0,S^*]\,.
$$
By construction, the zeros of $B$ represent  all stationary solutions to~\eqref{C}. In fact, besides the trivial equilibrium corresponding to $\mu=S^*$, there is exactly one other equilibrium, namely the nontrivial equilibrium corresponding to some~$\mu_*\in (0,\underline{\mu}]$:

\begin{prop}\label{L40}
Assume~\eqref{GA}, \eqref{gg} and~\eqref{NT}. Then, there is $\mu_*\in (0,\underline{\mu}]$ such that $\mu=S^*$ and $\mu=\mu_*$ are the only zeros of the function $B:[0,S^*]\to \R$.
\end{prop}

\begin{proof}
From $h(0)=h_*$, $S(h_*)=0$, $\nu(0)=0$, and $\partial_z {\rm c}(\cdot,0)\equiv 0$ we obtain
$$
B(0)=-abh_*-\frac{k_2}{\alpha+k_2}d(h_*)h_*<0
$$
while $h(S^*)=0$ and $\partial_z {\rm c}(0,\cdot)\equiv 0$ entail that
$B(S^*)=0$. Next, observe from~Lemma~\ref{L17} that ${\rm c}(z,\mu)\ge {\rm c}(0,\mu)=\mu$ for $z\ge 0$ and $\mu\in [0,S^*]$, hence, by the monotonicity of $r$,
$$
\kappa \partial_z {\rm c}\big(h(\mu),\mu\big)=\int_0^{h(\mu)}r\big({\rm c}(z,\mu)\big)\,\rd z\ge r(\mu)h(\mu)
$$
for $\mu\in [0,S^*]$. Therefore,
\begin{align*}
B(\mu)\ge -h(\mu) M(0,\mu)>0\,,\quad \mu \in (\underline{\mu},S^*)\,,
\end{align*}
by Lemma~\ref{L22}. Finally, we show that $B$ is strictly increasing on $[0,\underline{\mu})$. To this end, we first differentiate~\eqref{id} with respect to $\mu$ and use Lemma~\ref{L17} to derive that
\begin{align}\label{stst}
\big[\partial_z{\rm c}(h(\mu),\mu)-S'(h(\mu))\big] h'(\mu)=-\partial_\mu{\rm c}(h(\mu),\mu)=-w(h(\mu),\mu)\le -1
\end{align}
for $\mu\in [0,S^*]$. Using the identity~\eqref{stst} and \eqref{A1ss}, we obtain that
\begin{align*}
\big[\partial_z{\rm c}(h(\mu),\mu)-&S'(h(\mu))\big]\frac{\rd}{\rd\mu}B(\mu)\\
&=-a\big[ r\big({\rm c}(h(\mu),\mu)\big)-b\big] w\big(h(\mu),\mu\big)\\
&\quad +a\kappa \big[\partial_z{\rm c}(h(\mu),\mu)-S'(h(\mu))\big]\partial_zw\big(h(\mu),\mu\big) \\
&\quad +\frac{k_2-\nu\big(S(h(\mu))\big)}{\alpha+k_2-\nu\big(S(h(\mu))\big)}\big[d'\big(h(\mu)\big)h(\mu)+d\big(h(\mu)\big)\big]w\big(h(\mu),\mu\big)\\
&\quad -\alpha\frac{\nu'\big(S(h(\mu))\big)S'(h(\mu))}{\big(\alpha+k_2-\nu\big(S(h(\mu))\big)\big)^2}d\big(h(\mu)\big)h(\mu)w\big(h(\mu),\mu\big)\\
&\ge -a\big[ r\big({\rm c}(h(\mu),\mu)\big)-b\big] w\big(h(\mu),\mu\big)\\
&\quad +a\kappa \big[\partial_z{\rm c}(h(\mu),\mu)-S'(h(\mu))\big]\partial_zw\big(h(\mu),\mu\big) \\
&\quad +\frac{k_2-\nu\big(S(h(\mu))\big)}{\alpha+k_2-\nu\big(S(h(\mu))\big)}d\big(h(\mu)\big) \\
&= M\big(h(\mu),\mu)\,,
\end{align*}
where we used for the inequality that $\nu$ and $d$ are increasing and that $S$ is decreasing along with \eqref{H4} to drop the last term and part of the second-to-last term. Therefore, since $\partial_z{\rm c}(h(\mu),\mu)-S'(h(\mu))>0$ for $\mu\in [0,S^*]$ by \eqref{stst} and the monotonicity of $h$, we conclude from~\eqref{M} that $B'(\mu)>0$ for $\mu\in [0,\underline{\mu})$. 

Consequently, the function $B$ is strictly increasing on $[0,\underline{\mu})$, satisfies $B(0)<0$ while $B(\mu)>0$ for $\mu \in (\underline{\mu},S^*)$. It thus follows that there exists $\mu_*\in (0,\underline{\mu}]$ such that $\mu_*$ is the only zero of $B$ in the interval $[0,S^*)$ as claimed.
\end{proof}

Theorem~\ref{T5} is now a consequence of Proposition~\ref{L40} and the previous considerations.

\subsection{Stability of the Nontrivial Equilibrium}

Finally, we provide conditions that ensure the local asymptotic stability of the unique nontrivial equilibrium
$$
(h_\star,S_\star,Q_\star)\in (0,\infty)\times (0,S^*)\times (0,\infty)\,,\qquad  c_\star\in {\rm C}^2([0,h_\star])\big)
$$ 
to \eqref{C} established in the previous Subsection~\ref{SS4.2}. To this end we impose the same assumptions as therein but strengthen condition~\eqref{h3a} to 
\begin{equation}\label{hh3}
\max\left\{\frac{1}{2a},\frac{b}{\alpha+k_2},\frac{\alpha}{a(\alpha+k_2)}\right\}\le \frac{\beta k_1\kappa}{\rho}\le \frac{2\alpha}{a\nu(S^*)}\,.
\end{equation}
Again we point out that condition~\eqref{hh3}  constitutes structural assumption that is largely independent of the magnitude of the inlet concentration $S^*$.
We derive the local stability result:

\begin{thm}\label{T6}
Assume~\eqref{GA}, \eqref{gg},~\eqref{h1},~\eqref{h2}, and~\eqref{hh3}.
Then, the unique nontrivial equilibrium solution 
$$
(h_\star,S_\star,Q_\star)\in (0,\infty)\times (0,S^*)\times (0,\infty)
$$ 
to \eqref{C} with $c_\star\in {\rm C}^2([0,h_\star])\big)$ provided by Theorem~\ref{T5} is locally asymptotically  stable in~$\R^3$.
\end{thm}

\begin{proof}
Observing that~\eqref{hh3} implies~\eqref{h3a}, Theorem~\ref{T5} ensures the existence and uniqueness of the nontrivial equilibrium.
Recall then from~\eqref{h1} and~\eqref{h2} that
\begin{equation}\label{Delta}
\Delta:=\alpha+k_2-\nu(S_\star)>\alpha+k_2-\nu(S^*)>\alpha>0\,. 
\end{equation}
As in Section~\ref{Sec2} we use dimensionless variables (see~\eqref{dvar}) and write
$u(y)=c\bigl(yh\bigr)$, in particular $u_\star(y)=c_\star\bigl(yh_\star\bigr)$. In fact, from Proposition~\ref{P1} we recall that $[h,S]\mapsto u[h,S]$ is of class $C^1$ and $u_\star=u_\star[h_\star,S_\star]$.
Define $f$ as the right-hand side of the ODE-system \eqref{H3}-\eqref{H5} as in~\eqref{f}.
Under assumption~\eqref{gg} we then compute the Jacobian $A:=\partial f(h_\star,S_\star,Q_\star)$ of $f$  at the equilibrium $(h_\star,S_\star,Q_\star)$ as
\begin{equation}\label{MA}
A=
\left(\begin{matrix}
a   h_\star   H -\frac{\alpha  d(h_\star)}{\Delta}-d'(h_\star)  h_\star\quad  & a   h_\star   G\quad & \frac{\alpha}{\beta}\\[4pt]
-\frac{\rho (\Delta-\alpha)d(h_\star)}{a   \kappa    \Delta}-\frac{b   \rho}{\kappa}-\frac{\rho  h_\star}{\kappa}  H \quad &-D-\frac{k_1   \nu'(S_\star)  \beta  d(h_\star)   h_\star}{\Delta}-\frac{\rho  h_\star}{\kappa}  G\quad  & -k_1   \nu(S_\star)\\[4pt]
\beta  d(h_\star) +\beta  d'(h_\star)  h_\star\quad &\frac{\beta  d(h_\star)   h_\star   \nu'(S_\star)}{\Delta}&\quad -\Delta
\end{matrix}\right)\,,
\end{equation}
where
\begin{equation}\label{GH}
H:=\int_0^1 r'(u_\star(y))\partial_h u_\star(y)\,\rd y\le 0\,,\qquad G:=\int_0^1 r'(u_\star(y))\partial_S u_\star(y)\,\rd y> 0
\end{equation}
according to Corollary~\ref{Cor1}.  In order to show that all eigenvalues of $A$ have a negative real part under assumptions~\eqref{h1},~\eqref{h2}, and~\eqref{hh3} one may apply the Routh-Hurwitz Criterion. We postpone the tedious details to the Appendix~\ref{App}. 
We thus conclude that the nontrivial equilibrium~$(h_\star,S_\star,Q_\star)$ is indeed locally asymptotically  stable. 
\end{proof}

One may expect the nontrivial equilibrium to be \emph{globally} asymptotically stable in $\mathbb{R}^3$ (at least in a parameter regime compatible with Theorem~\ref{T6}). This conjecture is supported by the numerical simulations depicted in Figure~\ref{fig:case3} (see also~\cite{MasicEberl12}) and is consistent with Proposition~\ref{PX}.

\begin{appendix}

\nequation
\aequation
 
\section{Numerical Simulations}\label{AppN}
We present some basic numerical simulations of system~\eqref{H} illustrating the analytical results of Sections~\ref{Sec3} and~\ref{Sec4}. The simulations were performed in Mathematica.

Throughout, we adopt Monod kinetics, linear net growth ($a=1$ in \eqref{gg}), and linear detachment:
\begin{equation*}
r(c)=\frac{\mu_r \, c}{K_r+c}\,,\qquad
\nu(S)=\frac{\mu_\nu \, S}{K_\nu+S}\,,\qquad
g(r)=r-b\,,\qquad
d(h)=d_0 h\,.
\end{equation*}
These choices satisfy the structural assumptions~\eqref{GA} and~\eqref{gg}. We use the parameters in Table~\ref{tab:params} for illustrative purposes our analytical results only. We refer to \cite{MasicEberl12} for biologically calibrated values.
\\[0.3cm]

\begin{table}[H]
\centering
\renewcommand{\arraystretch}{1.3}
\begin{tabular}{clc}
\hline
\textbf{Symbol} & \textbf{Description} & \textbf{Value}\\
\hline
$\mu_r$ &  maximum specific growth rate (biofilm) & $4$\\
$K_r$ &  half-saturation constant (biofilm) & $1$\\
$\mu_\nu$ & maximum specific growth rate (planktonic) & $2$\\
$K_\nu$ & half-saturation constant (planktonic) & $1$\\
$b$ & biofilm death rate in $g(r)=r-b$ & $2$\\
$d_0$ & detachment coefficient in $d(h)=d_0 h$ & $1$\\
$\kappa$ & diffusion coefficient & $1$\\
$D$ & dilution rate & $1$\\
$k_1$ & substrate consumption parameter & $1$\\
$k_Q$ & planktonic death rate & $1$\\
$\alpha$ & attachment rate & $1$\\
$\rho$ & biomass density parameter & $1$\\
$\beta$ & conversion factor & $1$\\
$k_2$ & $=D+k_Q$ \quad (derived) & $2$\\
\hline
\end{tabular}

\vspace{0.2cm}

\caption{Model parameters used in the numerical simulations.}
\label{tab:params}
\end{table}

\newpage

\begin{figure}[H]
\centering
\includegraphics[width=0.62\textwidth]{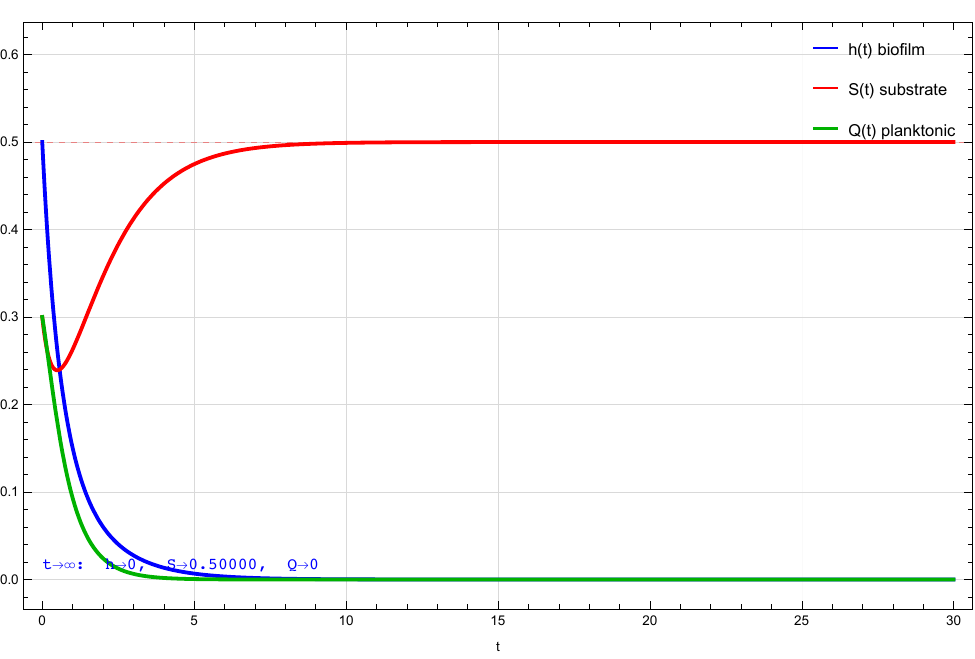}
\caption{Convergence to the washout equilibrium $(0,\,S^*,\,0)=(0,\,0.5,\,0)$. 
The conditions of Corollary~\ref{C10} are satisfied with $d(0)=0$, $g(r(S^*))=-2/3<0$, and $\nu(S^*)=2/3<k_2=2$. Initial conditions: $(h_0,S_0,Q_0)=(0.5,\,0.3,\,0.3)$.}
\label{fig:case1}
\end{figure}

\begin{figure}[H]
\centering
\includegraphics[width=0.62\textwidth]{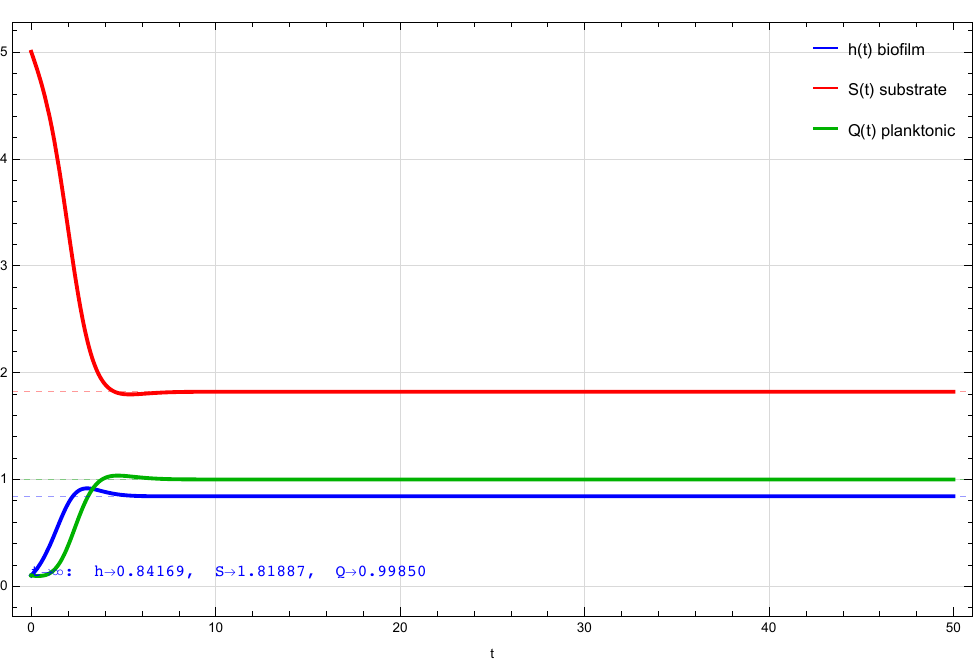}
\caption{Convergence to the unique nontrivial equilibrium for $S^*=5$.
The washout equilibrium is unstable since $g(r(S^*))=4/3>0$ (cf.\ Remark~\ref{R1} with $d(0)=0$). Existence and uniqueness follow from Theorem~\ref{T5} with $\alpha+k_2=3>\nu(S^*)=5/3$.
Initial conditions: $(h_0,S_0,Q_0)=(0.1,\,5.0,\,0.1)$.
Dashed lines indicate the numerically determined equilibrium $(h_\star,S_\star,Q_\star)\approx (0.923,2.118,0.518)$.}
\label{fig:case2}
\end{figure}

\begin{figure}[H]
\centering
\includegraphics[width=\textwidth]{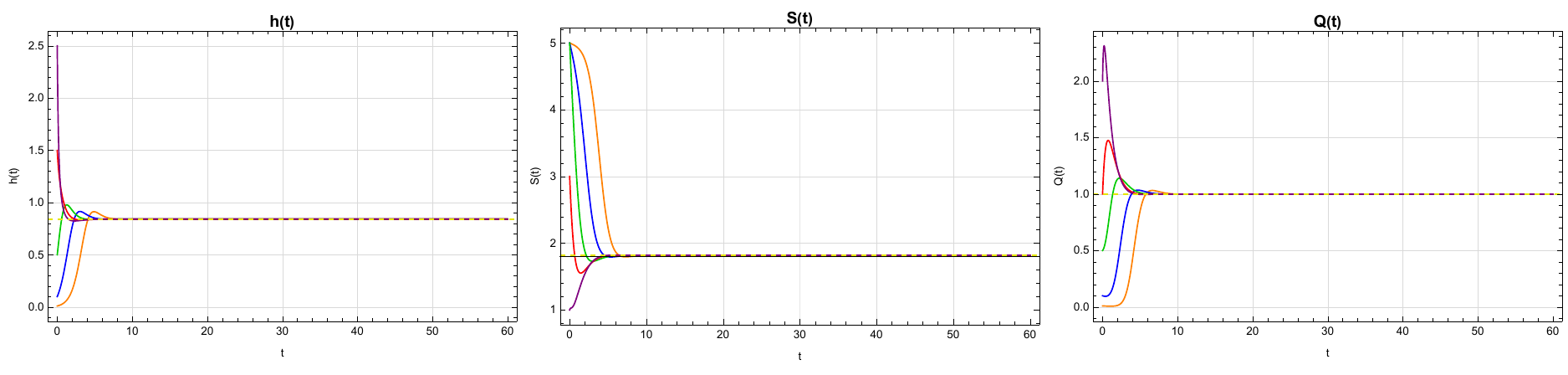}
\caption{Five trajectories from widely separated initial conditions in $(0,\infty)^3$ all converge to the unique nontrivial equilibrium $(h_\star,S_\star,Q_\star)$ (gold dashed lines). This is consistent with the boundedness of orbits in Proposition~\ref{PX} and provides numerical evidence that the locally stable equilibrium of Theorem~\ref{T6} is in fact globally attracting. Parameters  and nontrivial equilibrium as in Figure~\ref{fig:case2} with $S^*=5$.}
\label{fig:case3}
\end{figure}

\medskip

\section{Proof of Theorem~\ref{T6}}\label{App}

To complete the proof of Theorem~\ref{T6} it remains to check that all eigenvalues of the $(3\times 3)$-matrix $A=(a_{ij})_{1\le i,j\le 3}$ in \eqref{MA} have negative real parts. According to the Routh-Hurwitz-Criterion this is the case provided that
\begin{align}\label{mpos}
m_0>0\,,\quad m_1>0\,,\quad m_2>0\,,\quad  m_1m_2-m_0>0\,,
\end{align}
where
\begin{align*}
&m_0:=-\det(A)\,,\quad m_1:=-\mathrm{trace}(A)\,,\\
&m_2:=\det\left(\begin{matrix}
a_{11} & a_{12} \\
a_{21} & a_{22}
\end{matrix}\right)
+\det\left(\begin{matrix}
a_{11} & a_{13} \\
a_{31} & a_{33}
\end{matrix}\right)
+\det
\left(\begin{matrix}
a_{22} & a_{23} \\
a_{32} & a_{33}
\end{matrix}\right)\,.
\end{align*}
We check the positivity of each term in~\eqref{mpos} using that
\begin{equation}\label{pos}
H\le 0\,,\quad G>0\,,\quad \nu(S^*)\ge \nu(S_\star)\,,\quad \nu'(S_\star)\ge 0\,,\quad d'(h_\star)\ge 0
\end{equation}
due to~\eqref{GH} and the monotonicity assumptions on $\nu$ and $h$.
\\

\noindent{\bf (i)} As for $m_0$ one computes
\begin{align*}
m_0=&\Big(-a \Delta  H+ \big[\Delta- \alpha\big] d'(h_\star) \Big)h_\star D-\frac{\big[a \beta  k_1  \kappa (\alpha+k_2) -\alpha \rho \big]}{\Delta \kappa } h_\star^2 \nu'(S_\star) d(h_\star)  H\\
&+\left(\frac{\left( \rho \big[\Delta- \alpha\big]+a\beta k_1  \kappa   \nu(S_\star)\right)\left(  d(h_\star)+h_\star d'(h_\star)\right)}{ \kappa }+\frac{a b \rho \Delta }{\kappa}\right)h_\star G\\
&+\frac{\rho \big[\Delta- \alpha\big]+ a \beta k_1  \kappa \nu(S_\star) }{\Delta^{2} \kappa a}\alpha h_\star  \nu'(S_\star) d(h_\star) ^{2}+\frac{k_2}{\Delta} \beta k_1  h_\star^2   \nu'(S_\star) d'(h_\star)  d(h_\star)\\
&+\frac{ \nu'(S_\star) \alpha \rho b h_\star}{\Delta \kappa} d(h_\star) 
\end{align*} 
and thus $m_0>0$ since each term is nonnegative owing to~\eqref{pos} and using~\eqref{Delta} and~\eqref{hh3} for the terms in square brackets.\\

\noindent{\bf (ii)} As for $m_1$ one computes
\begin{align*}
m_1=-a h_\star H+\frac{\alpha d(h_\star) }{\Delta}+d'(h_\star)  h_\star+D+\frac{\beta k_1   \nu'(S_\star)  d(h_\star)  h_\star}{\Delta}+\frac{\rho h_\star }{\kappa}G+\Delta
\end{align*} 
and thus $m_1>0$ due to~\eqref{pos} and~\eqref{Delta}.\\

\noindent{\bf (iii)} As for $m_2$ one computes
\begin{align*}
m_2=& \left(-a h_\star H+\frac{\alpha d(h_\star) }{\Delta}+d'(h_\star)  h_\star+\Delta\right) D+
\left(\frac{ d(h_\star)+h_\star  d'(h_\star)+a b   + \Delta}{\kappa}\right) \rho h_\star G\\
&-\left(\frac{ \beta k_1  h_\star   \nu'(S_\star) d(h_\star) }{\Delta}+ \Delta \right)ah_\star H+\frac{\alpha \beta  k_1  h_\star  \nu'(S_\star) d(h_\star) ^{2}}{\Delta^{2}}\\
&+\frac{\beta k_1  \nu'(S_\star)   h_\star^{2} d(h_\star) d'(h_\star) }{\Delta}+\frac{(\alpha+k_2)   \beta  k_1\nu'(S_\star) h_\star d(h_\star)}{\Delta} +\big[\Delta- \alpha\big]  h_\star d'(h_\star) 
\end{align*} 
and thus $m_2>0$ since each term is nonnegative owing to~\eqref{pos} and~\eqref{Delta}.\\

\noindent{\bf (iv)} Finally,  for $m_3:=m_1m_2-m_0$ one computes
\begin{align*}
m_3=&\left(-a h_\star H+\frac{\alpha d(h_\star) }{\Delta}+d'(h_\star)  h_\star +\Delta\right) D^{2}\\
&+\left(-a h_\star H+\frac{(\Delta+  \alpha)  d(h_\star) }{\Delta }+2 h_\star  d'(h_\star) + ab +2    \Delta\right)\frac{\rho D h_\star }{\kappa}G\\
&+ a^{2}D h_\star^{2} H^{2}-2a\left(\left(\beta h_\star k_1    \nu'(S_\star)+ \alpha  \right) \frac{d(h_\star)}{\Delta} + h_\star  d'(h_\star)  + \Delta \right)h_\star D H\\
&+\Big(2 \beta  h_\star k_1   \nu'(S_\star)+\alpha\Big)\frac{\alpha D d(h_\star)^{2}}{\Delta^{2}}\\
& +\left(2\left(\beta h_\star k_1   \nu'(S_\star)+ \alpha\right) \frac{d'(h_\star)}{\Delta} +\frac{(2     \Delta+   \nu(S_\star) ) \beta k_1  \nu'(S_\star)}{\Delta}\right)  h_\star d(h_\star)D+2 \alpha  d(h_\star) D\\
& +D d'(h_\star)^{2} h_\star^{2}+2 D\Delta h_\star d'(h_\star) +D\Delta^{2}+\Big(d(h_\star) +h_\star  d'(h_\star) +ab+ \Delta\Big) \frac{\rho^{2} h_\star^2}{\kappa^{2}} G^{2}\\
&-\left(\left(\frac{\beta k_1  h_\star   \nu'(S_\star)}{\Delta}+1\right) d(h_\star) +h_\star  d'(h_\star) +ab +2 \Delta\right) \frac{a \,h_\star^{2} \rho}{\kappa} G H\\
&+\left(\frac{\beta k_1 h_\star\left( \Delta+\alpha\right)  \nu'(S_\star)}{\Delta}+\alpha\right) \frac{h_\star  \rho d(h_\star)^{2}}{\Delta \kappa } G\\
&+\Big(2 \beta k_1 h_\star   \nu'(S_\star)+ \Delta+ \alpha\Big) \frac{   \rho  h_\star^{2} d'(h_\star)d(h_\star)}{\Delta \kappa}G \\
&+\Big(2  \Delta+ab+  \nu(S_\star)\Big)  \frac{\rho \beta k_1 h_\star^2\nu'(S_\star)d(h_\star)}{\Delta \kappa} G\\
&+\Big( \left[2\alpha\rho-a\beta k_1 \kappa\nu(S_\star)\right] \Delta +\alpha\rho ab\Big) \frac{h_\star d(h_\star)}{\Delta \kappa} G+ \frac{\rho h_\star^{3}  d'(h_\star)^{2}}{\kappa}G\\
 & +\Big( \left[2\Delta\rho-a\beta k_1 \kappa\nu(S_\star)\right] +ab \rho\Big) \frac{h_\star^2 d'(h_\star) }{ \kappa}  G
+\frac{\rho  h_\star\Delta^{2}}{\kappa} G\\
&+\left(\frac{ \beta  k_1  h_\star  \nu'(S_\star) d(h_\star) }{\Delta}+\Delta\right) a^{2}h_\star^{2}   H^{2}\\
&-\Big(\beta k_1 h_\star \nu'(S_\star) +2\alpha\Big) \frac{\beta  k_1 a h_\star^{2}  \nu'(S_\star)d(h_\star) ^{2}}{\Delta^{2}} H\\
&-  \Big(2\big(h_\star    d'(h_\star) +\Delta\big)
a  \beta k_1 \kappa +\alpha\rho\Big)\frac{h_\star^2\nu'(S_\star) d(h_\star) }{\Delta\kappa} H\\
& -\left(\big[2 \Delta-\alpha\big] h_\star d'(h_\star) + \Delta^{2} +\alpha  d(h_\star)\right) a h_\star H\\
& +\big(\beta  k_1 h_\star \nu'(S_\star) +\alpha\big)  \frac{ \alpha \beta k_1  h_\star   \nu'(S_\star) d(h_\star)^{3}}{\Delta^{3}} \\
&+\big(\beta k_1 h_\star  \nu'(S_\star)+2\alpha\big)\frac{\beta k_1 h_\star^{2}   \nu'(S_\star)}{\Delta^{2}} d'(h_\star)d(h_\star)^{2} +\big(\Delta+\nu(S_\star)\big)\frac{\beta^{2}k_1^{2} h_\star^{2}  \nu'(S_\star)^{2}}{\Delta^{2}}d(h_\star)^{2}\\
&+\Big(\big[2a\beta k_1\kappa-\rho\big]\Delta+\alpha \rho\Big)\frac{\alpha  h_\star\nu'(S_\star)}{\Delta^{2} \kappa a} d(h_\star)^{2}+\frac{\beta k_1   h_\star^{3}  \nu'(S_\star) d'(h_\star) ^{2}}{\Delta} d(h_\star)\\
&+\left(2 \beta k_1  \nu'(S_\star)  \,h_\star^{2}+\big[\Delta-\alpha\big]\frac{\alpha h_\star}{\Delta }\right) d'(h_\star)d(h_\star) \\
& +\Big[\Delta  \beta  k_1  \kappa(\alpha+k_2)-\alpha b \rho\Big]\frac{ h_\star  \nu'(S_\star) d(h_\star)}{\Delta \kappa} +\big[\Delta-\alpha\big] \Big(h_\star^{2}  d'(h_\star)^{2} +\Delta h_\star d'(h_\star)\Big)
\end{align*} 
and thus $m_1m_2-m_0>0$ since each term is nonnegative, where we use~\eqref{pos} and --~for the terms in square brackets --~\eqref{hh3} and~\eqref{Delta}.

\end{appendix}

\bibliographystyle{siam} 
\bibliography{BiofilmLiterature}

\end{document}